\documentclass[a4paper,oneside,11pt]{article}

\usepackage{amsmath,amsfonts,amscd,amssymb,amsthm}
\usepackage{longtable,geometry}
\usepackage[english]{babel}
\usepackage[utf8]{inputenc}
\usepackage[active]{srcltx}
\usepackage[T1]{fontenc}
\usepackage{graphicx}
\usepackage{pstricks}
\usepackage{bbm}
\usepackage{enumitem}
\usepackage{color}
\usepackage{MnSymbol}
\usepackage{stmaryrd}
\usepackage{nicefrac}
\usepackage{calrsfs}
\usepackage{mathtools} 
\usepackage{mdwlist}   
\usepackage{epstopdf} 

\newtheorem{theorem}{Theorem}[section]

\newtheorem{claim}{Claim}
\newtheorem*{claim*}{Claim}

\newtheorem{corollary}[theorem]{Corollary}
\newtheorem{lemma}[theorem]{Lemma}
\newtheorem{proposition}[theorem]{Proposition}

\newtheorem{remark}[theorem]{Remark}

\newcommand{\be}[1]{\begin{equation}\label{#1}}
\newcommand{\ee}{\end{equation}}
\numberwithin{equation}{section}

\newcommand{\ba}[1]{\begin{align}\label{#1}}
\newcommand{\ea}{\end{align}}
\numberwithin{equation}{section}

\newcommand{\ben}{\begin{equation*}}
\newcommand{\een}{\end{equation*}}
\numberwithin{equation}{section}

\renewenvironment{proof}[1][\relax]
{\paragraph{Proof\ifx#1\relax\else~of #1\fi}}%
{~\hfill$\square$\par\bigskip}

\newcommand{\calA}{\mathcal{A}}
\newcommand{\calB}{\mathcal{B}}
\newcommand{\calC}{\mathcal{C}}
\newcommand{\calD}{\mathcal{D}}
\newcommand{\calE}{\mathcal{E}}
\newcommand{\calF}{\mathcal{F}}
\newcommand{\calG}{\mathcal{G}}
\newcommand{\calH}{\mathcal{H}}
\newcommand{\calI}{\mathcal{I}}

\newcommand{\calO}{\mathcal{O}}

\newcommand{\calS}{\mathcal{S}}

\newcommand{\calW}{\mathcal{W}}
\newcommand{\calX}{\mathcal{X}}
\newcommand{\calY}{\mathcal{Y}}

\newcommand{\bbN}{\mathbb{N}}

\newcommand{\bbR}{\mathbb{R}}

\newcommand{\bbZ}{\mathbb{Z}}

\newcommand{\ga}{\gamma}
\newcommand{\de}{\delta}
\newcommand{\De}{\Delta}
\newcommand{\ep}{\varepsilon}
\newcommand{\eps}{\ep}
\newcommand{\om}{\omega}

\newcommand{\si}{\sigma}

\renewcommand{\Im}{\textrm{Im}}

\newcommand{\rk}[1]{\bgroup\color{red}%
\par\medskip\hrule\smallskip%
\noindent\textbf{#1}%
\par\smallskip\hrule\medskip\egroup}

\title{The phase transitions of the random-cluster and Potts models on slabs with $q \geq 1$ are sharp}
\author{Ioan Manolescu and Aran Raoufi}
\date{\today}

\usepackage{wasysym}

\newcommand{\Lat}{\calG}
\newcommand{\slab}{\calS}
\newcommand{\bfS}{\mathsf{S}}
\newcommand{\bfR}{\mathsf{R}}
\newcommand{\Ball}{\Lambda}
\newcommand{\pd}{\partial}

\newcommand\lra{\leftrightarrow}
\newcommand\xlra{\xleftrightarrow}

\usepackage{psfrag}
\def\mik{1}
\newcommand\cpsfrag[2]{\ifnum\mik=1\psfrag{#1}{#2}\fi}

\newcommand{\ovr}{\overline}
\newcommand{\La}{\Lambda}

\begin{document}

\maketitle

\begin{abstract}
We prove sharpness of the phase transition for the random-cluster model with $q \geq 1$ 
on graphs of the form $\slab := \Lat \times S$, 
where $\Lat$ is a planar lattice with mild symmetry assumptions, and $S$ a finite graph.
That is, for any such graph and any $q \geq 1$, there exists some parameter $p_c = p_c(\slab, q)$,
below which the model exhibits exponential decay 
and above which there exists a.s. an infinite cluster. The result is also valid for the random-cluster model on planar graphs with long range, compactly supported interaction.
It extends to the Potts model via the Edwards-Sokal coupling.
\end{abstract}

\section{Introduction}

In the last few years, a variety of results concerning the phase transition of the random-cluster model (or FK-percolation)
on planar lattices have emerged; see~\cite{BefDum12,duminil2015continuity, DumMan13, DumMan14}.
The first are specific to the self-dual setting of the square lattice, 
while the third, still in preparation, extends the results of the first two to isoradial graphs. 
Finally, the fourth paper -- a companion to the present paper -- 
proves the sharpness of the phase transition of random-cluster models 
on generic planar graphs with sufficient symmetry. 

These recent advances offer an understanding of planar random-cluster models 
that approaches that of Bernoulli percolation. 
However, contrary to the case of Bernoulli percolation, 
for which exponential decay in the subcritical phase was proved for lattices of any dimension  
(see~\cite{AizBar87,Men86} and~\cite{DumTas15} for a recent short proof),
the phase transition of the random-cluster model in dimensions above two is still not known to be sharp. 
We take a first step in this direction by proving the result for slabs, 
that is finite planar "slices" of a d-dimensional lattice.

Percolation on slabs has already been considered in the literature, most notably in the paper \cite{GriMar90}, 
where it was shown that the critical point of percolation on $\bbZ^2 \times \{0,1,2,\dots,N\}^{(d-2)}$ 
tends decreasingly to that of $\bbZ^d$ as $N \to \infty$. 
There is work in progress on the same type of result for the random-cluster model with integer $q$ \cite{DumTasGMRC}.
Let us also mention that Bernoulli percolation on slabs has recently been shown to exhibit a continuous phase transition~\cite{duminil2014absence};
a result which is also long sought for lattices of general dimension. 
Arguments similar to those of \cite{DumTasGMRC} also appear in \cite{NTW15} and \cite{BasSap16}.

The present paper blends the method of~\cite{DumMan14} 
with the techniques of~\cite{duminil2014absence}. 
It is intended as a complement to~\cite{DumMan14}, 
focusing essentially on the new elements needed to treat the case of slabs.

Next we briefly introduce the model. 
For more details on the random-cluster model, we refer the reader to the monograph~\cite{Gri06}.

Consider a finite graph $G = (V_G,E_G)$. 
The random-cluster measure with edge-weight $p\in [0,1]$ and cluster-weight $q > 0$ on $G$ 
is a measure $\phi_{p,q,G}$ on configurations $\omega\in\{0,1\}^{E_G}$. 
For such a configuration $\om$, 
an edge $e$ is said to be {\em open} (in $\omega$) if $\omega(e)=1$, otherwise it is {\em closed}. 
The configuration  $\omega$ can be seen as a subgraph of $G$ with vertex set $V_G$ and edge-set $\{e\in E_G:\omega(e)=1\}$. 
A {\em cluster} is a connected component of the subgraph $\omega$. 
Let $o(\omega)$, $c(\omega)$ and $k(\omega)$ denote the number of open edges, closed edges and clusters  in $\omega$, respectively. 
The probability of a configuration is then equal to
\begin{equation*}
\phi_{p,q,G}(\omega)=\frac{p^{o(\omega)}(1-p)^{c(\omega)}q^{k(\omega)}}{Z(p,q,G)},
\end{equation*}
where $Z(p,q,G)$ is a normalising constant called the partition function. 

Fix for the rest of the paper a connected planar locally-finite graph $\Lat = (V_{\Lat}, E_{\Lat})$, 
which is invariant under the action of some lattice $\Lambda\simeq\bbZ\oplus\bbZ$, 
under reflection with respect to the line $\{(0,y), y \in \bbR\}$
and rotation by some angle $\theta \in (0,\pi)$ around $0$.
For simplicity we will assume in the present paper that $\theta = \pi/2$
and that $\Lat$ is invariant under translations by the vectors $(1,0)$ and $(0,1)$.

In addition let $S = (V_S, E_S)$ be a finite graph and define the "slab" $\slab = \Lat \times S$.
That is $\slab$ is the graph with vertices $V_{\slab} = V_{\Lat} \times V_S$ and edges $E_{\slab}$
connecting two vertices $(u,v) \in V_{\slab}$ and $(u',v')\in V_{\slab}$ if 
either $u = u'$ and $(v,v') \in E_S$ or $(u,u')\in E_{\Lat}$ and $v = v'$.  
Maybe the most common such example is for $\Lat = \bbZ^2$ and $S = \{1,\dots,n\}$, 
in which case $\slab$ is a slice of thickness $n$ of the three dimensional lattice $\bbZ^3$. 

For $p \in [0,1]$ and $q \geq 1$, random-cluster measures with parameters $p, q$ may be defined on the infinite graph $\slab$ 
by taking weak limits of measures on sequences of nested finite graphs $G_n$ tending to $\slab$
(see~\cite[Ch. 4]{Gri06} or~\cite[Sec 4.5]{Dum13} for a detailed account).
We call such limits \emph{infinite-volume} measures.
For a pair of parameters $p,q$, more than one such infinite-volume measure may exist;
the two most notable infinite-volume measures are the free and wired ones, 
denoted by $\phi_{p,q,\slab}^0$ and $\phi_{p,q,\slab}^1$, respectively.
These are ordered in that, for $p < p'$ and $q\geq 1$, 
$$ \phi_{p,q,\slab}^0 \leq_{\rm st} \phi_{p,q,\slab}^1 \leq_{\rm st} \phi_{p',q,\slab}^0,$$
where $\leq_{\rm st}$ denotes stochastic domination. 
Moreover, $\phi_{p,q,\slab}^0$ and  $\phi_{p,q,\slab}^1$ are the extremal measures with parameters $p$ and $q$, 
in the sense that, if $\phi_{p,q,\slab}$ is an infinite volume measure with these parameters, then 
$$ \phi_{p,q,\slab}^0 \leq_{\rm st} \phi_{p,q,\slab} \leq_{\rm st} \phi_{p,q,\slab}^1.$$
While it is possible to have values of $p$ for which the infinite volume measure is not unique, i.e. for which 
$\phi_{p,q,\slab}^0 \neq \phi_{p,q,\slab}^1$, only at most countably many such values of $p$ exist for any fixed $q \geq 1$. 
For $p,q$ such that $\phi_{p,q,\slab}^0 = \phi_{p,q,\slab}^1$, we will denote the unique infinite-volume measure by~$\phi_{p,q,\slab}$.

\begin{theorem}\label{thm:main}
	Fix $q\ge1$. There exists $p_c=p_c(\slab) \in [0,1]$ such that
	\begin{itemize}[nolistsep,noitemsep]
	\item for $p < p_c$, there exists $c=c(p,\slab)>0$ such that for any $x,y \in \slab$, 
	\begin{align}
	\phi_{p,q,\slab}^1[x\text{ and }y\text{ are connected by a path of open edges}]\le \exp(-c|x-y|), \label{eq:exp_decay0}
	\end{align}
	\item for $p > p_c$, there exists a.s. an infinite open cluster under $\phi_{p,q, \slab}^0$.
	\end{itemize}
\end{theorem}

The equivalent of Theorem~\ref{thm:main} is also valid for planar random-cluster models with finite range interactions; 
we define these next. 
Let $J: V_\Lat \times V_\Lat \rightarrow [0,+\infty)$ be a function 
with the property that there exists a constant $M \geq 1$ such that, if $\rm{d}_{\Lat} (x,y) >M$, then $J(x,y)=0$
(where $\rm{d}_{\Lat}$ is graph distance on $\Lat$). 
Moreover, suppose that $J$ has the same symmetries as $\Lat$.
Infinite-volume random-cluster measures $\phi_{\beta, q, \Lat, J}$ with parameters $\beta > 0$ and $q \geq 1$ 
may be defined as before as weak limits of measures $\phi_{\beta, q, G_n, J}$
on sequences of finite subgraphs $G_n$  tending toward $\Lat$, where $\phi_{\beta, q, G_n, J}$ is defined as
$$
	\phi_{\beta, q, G_n, J} (\om) = 
	\frac{\left( \prod_{ x,y \in V_{G_n}} (e^{\beta J(x,y)}-1) ^{\om(e)} \right) q^{k(\om)}}{Z(\beta, q, G_n, J)},
	$$ 
$Z(\beta, q, G_n, J)$ being a normalising constant. 
The same remarks about the different infinite-volume measures as in the case of slabs apply here. 


\begin{theorem}\label{thm:long-range}
	Fix $q\ge1$. There exists $\beta_c=\beta_c(\Lat, J) \in [0,1]$ such that
\begin{itemize}[nolistsep,noitemsep]
\item for $\beta < \beta_c$, there exists $c=c(p,\Lat, J)>0$ such that for any $x,y \in \Lat$,
\begin{align*}
\phi_{p,q,\Lat,J}^1[x\text{ and }y\text{ are connected by a path of open edges}]\le \exp(-c|x-y|), 
\end{align*}
\item for $\beta > \beta_c$, there exists a.s. an infinite open cluster under $\phi_{p,q, \Lat, J}^0$.
\end{itemize}
\end{theorem}

The proof of Theorem \ref{thm:long-range} is a direct adaptation of that of Theorem \ref{thm:main}.
In what follows we will only prove Theorem \ref{thm:main}; 
we leave the details of the adaptation of the proof to the second theorem to the interested reader. 

\paragraph{Results for the Potts model.} 
The above results have direct consequences for Potts model. 
Consider an integer $q\ge2$ and introduce the polyhedron $\Omega_q\subset\bbR^{q-1}$ 
with $q$ elements defined by the property that for any $a,b\in \Omega_q$
$$a\cdot b=\begin{cases}\ \ 1&\text{ if $a=b$,}\\
\ -\frac1{q-1}&\text{ otherwise,}\end{cases}$$
where $\cdot$ denotes the scalar product on $\bbR^{q-1}$.

Let $G = (V_G,E_G)$ be a finite graph and $\beta>0$.  
The {\em $q$-state Potts model on $G$ at inverse-temperature $\beta>0$ with free boundary conditions} 
is defined as follows. 
The {\em energy} of a configuration
$\sigma=(\sigma_x:x\in V_G)\in \Omega_q^{V_G}$ is given by the Hamiltonian
\begin{equation}\label{eq:def}
	H_G(\sigma)~:=~-\sum_{\{x, y\} \in E_G } \sigma_x\cdot \sigma_y.
\end{equation} 
The probability $\mu_{\beta, q, G}$ of a configuration $\sigma$ is defined by
\begin{align}\label{eq:Potts_def}
	\mu_{\beta, q, G}(\sigma)~:=~\frac{\exp[-\beta H_G(\sigma)]}{Z(G,\beta,q)},
\end{align}
where $Z(G,\beta,q)$ is defined in such a way that
the sum of the weights over all possible configurations equals 1. 

As for the random-cluster model, 
the q-state Potts measure with free boundary conditions $\mu_{\beta, q, \slab}$ on the infinite graph $\slab$ may be defined
by taking the weak limit of measures $\mu_{\beta, q, G_n}$ on sequences of nested finite graphs $G_n$ converging to $\slab$. 

The Edward-Sokal coupling between the measures $\phi_{p, q, \slab}^0$ and $\mu_{\beta, q, \slab}$ 
where $p= 1- \exp(-\frac{q}{q-1}\beta)$ yields the following relation for any two vertices $x, y \in \slab$
\begin{equation*} \label{eq:coupling}
\phi_{p, q, \slab}^0 [x\text{ and }y\text{ are connected by a path of open edges}]  = \mu_{\beta, q, \slab} (\sigma_x \cdot \sigma_y).
\end{equation*}
The above equation together with Theorem \ref{thm:main} imply the following corollary.
\begin{corollary}\label{thm:mainpotts}
	Fix $q\ge2$. There exists $\beta_c=\beta_c(\slab) \in [0, \infty) $ such that
\begin{itemize}[nolistsep,noitemsep]
\item for $\beta < \beta_c$, there exists $c=c(\beta,\slab)>0$ such that for any $x,y \in \slab$,
\begin{align*}
\mu_{\beta, q, \slab}  ( \sigma_x \cdot \sigma_y) \le \exp(-c|x-y|), 
\end{align*}
\item for $\beta> \beta_c$, there exists $c'=c'(\beta,\slab)>0$ such that for any $x,y \in \slab$,
\begin{align*}
\mu_{\beta, q, \slab}  ( \sigma_x \cdot \sigma_y) \geq c'.
\end{align*}    
\end{itemize}
\end{corollary}

Likewise, Theorem \ref{thm:long-range} may be translated for the Potts model. 
If $J$ is a function as before, define the Hamiltonian of the weighted Potts model on a finite sub-graph $G$ of $\Lat$ by
$$
	H_{G,J}(\sigma)~:=~-\sum_{x, y \in V_G } J(x,y) ~ \sigma_x\cdot \sigma_y, 
$$
and the associated measure $\mu_{\beta, q, G, J}$ by \eqref{eq:Potts_def}.
Infinite volume measures $\mu_{\beta, q, \Lat, J}$ may also be defined as above.

\begin{corollary}\label{thm:pottsshortrange}
	Fix $q\ge2$. There exists $\beta_c=\beta_c(\Lat, J) \in [0, \infty) $ such that
\begin{itemize}[nolistsep,noitemsep]
\item for $\beta < \beta_c$, there exists $c=c(\beta, \Lat, J)>0$ such that for any $x,y \in \Lat$,
\begin{align*}
\mu_{\beta, q, \Lat, J} ( \sigma_x \cdot \sigma_y ) \le \exp(-c|x-y|), 
\end{align*}
\item for $\beta> \beta_c$, there exists $c'=c'(\beta,\Lat, J)>0$ such that for any $x,y \in \Lat$,
\begin{align*}
\mu_{\beta, q, \Lat, J} ( \sigma_x \cdot \sigma_y) \geq c'.
\end{align*}    
\end{itemize}
\end{corollary}

We will not discuss further this adaptation to the Potts model.
For background on the Potts model and its coupling to the random-cluster model we direct the reader to \cite{Gri06}.
Deriving the two corollaries from Theorems \ref{thm:main} and \ref{thm:long-range} through the  Edward-Sokal coupling is straightforward.

\section{Notation and preparatory remarks}

\paragraph{Notation.}

In the rest of the paper $q \geq1$ will be fixed and we drop it from the notation. 
We will only work with infinite volume measures on $\slab$, hence we will equally drop $\slab$ from the notation for $\phi$.
\begin{center}
	\begin{minipage}{0.8\textwidth}
	\centering
		{\bf Thus $\phi_p$ will denote any infinite volume measure on $\slab$ with edge-weight $p$ and cluster-weight $q$. }
	\end{minipage}\\
\end{center}
It will be apparent in the proofs that we always allow ourselves to alter $p$ in a small open interval. 
We can therefore assume that all the values of $p$ mentioned hereafter are such that $\phi_p$ is the \emph{unique} infinite-volume measure. 

If $A$ is a subgraph of $\Lat$, then we define $\ovr{A}=A\times S$ and regard this as a subgraph of~$\slab$. 

Let $u,v \in \slab$ be two vertices, $D \subset\slab$ be a subgraph and $\om \in \{0,1\}^\slab$ be a configuration. 
We write $u \xlra{\om, D} v$ for the event that there exists an $\om$-open path, 
i.e. a self-avoiding chain of adjacent $\om$-open edges, linking $u$ and $v$ and contained in $D$. 
For sets $A, B$ of vertices of $\slab$, write $A \xlra{\om, D} B$ 
if there exists $u \in A$ and $v \in B$ such that $u \xlra{\om, D} v$ holds. 
When no confusion is possible, the configuration $\om$ will be omitted from the notation. 
If $D$ is omitted, it is assumed equal to~$\slab$. 

For $a<b$ and $c<d$, we identify $[a,b] \times [c,d]$ with the subgraph of $\Lat$ 
induced by the vertices contained in $[a,b] \times [c,d]$. 
We call a rectangle, a subgraph of $\slab$ of the form $R = \ovr{[a,b] \times [c,d]}$ 
(note that a rectangle is not planar, it has "thickness" $S$). 

For a rectangle $R = [a,b]\times [c,d]$, if we set $A=\ovr{\{a\}\times[c,d]}$ and $B=\ovr{\{b\}\times[c,d]}$ 
(respectively $A=\ovr{[a,b]\times\{c\}}$ and $B=\ovr{[a,b]\times\{d\}}$),
the event $A\xlra{\om, R} B$ is denoted by $\calC_h([a,b]\times[c,d])$ (respectively $\calC_v([a,b]\times[c,d])$)
and if it occurs we say that $R$ is crossed horizontally (respectively vertically). 
An open path from $A$ to $B$ is called a horizontal crossing (respectively vertical crossing).
When $a=0$ and $c=0$, we simply write $\calC_h(b,d)$ and $\calC_v(b,d)$ for the events above. 
When $b-a > d-c$, horizontal crossings are called crossings in the hard direction, 
while vertical ones are crossings in the easy direction. 
The terms are exchanged when $b-a < d-c$.

For $\ga = (\ga_1,\dots, \ga_m)$ and $\chi = (\chi_1,\dots, \chi_m)$ two paths of $\slab$, 
we say that $\ga$ and $\chi$ overlap at some point $g \in \Lat$ if 
there exist $i$ and $j$ such that $\ga_i, \chi_j \in \ovr{\{g\}}$.

For $g \in \Lat$, let $B_R(g)$ (and $\partial B_R(g)$) 
be the set of vertices at distance less than or equal to $R$ (equal to $R$, respectively) from $z$. 
For a point $z=(g,n) \in \slab$ define $\Lambda_R(z)= \ovr{B_R(g)}$ and $\partial \Lambda_R(z)= \ovr{\partial B_R(g)}$. 
We call $\Lambda_R(z)$ the box of size $R$ around $z$.

\paragraph{Strategy of the proof}

Define
\begin{align*}
p_c & := \inf\big\{p\in(0,1)~:~\phi_p(x\text{ is in an infinite cluster})>0\big\}\\
\tilde p_c & := \sup \big\{ p \in(0,1) ~:~ \lim_{n\rightarrow \infty}-\tfrac1n
\log\big[\phi_p(0 \lra \partial\Lambda_n)\big]>0\big\}.
\end{align*}
For $p < \tilde p_c$ we say that $\phi_p$ exhibits exponential decay since connection probabilities decay exponentially with the distance; 
for $p > p_c$, $\phi_p$ is supercritical, in that it contains a.s. an infinite cluster. 
It is immediate that $\tilde p_c\le p_c$. 
We wish to prove that $p_c=\tilde p_c$ (this is simply another way of stating the main result), 
and we therefore focus on the inequality $\tilde p_c\ge p_c$. 

As mentioned before, we adapt the argument of~\cite{DumMan14}, which consists of three steps:
\begin{itemize*}
\item  
First it is proved that, for $p > \tilde p_c$, 
the crossing probabilities under $\phi_{p}$  of $2n\times n$ rectangles in the easy direction 
are bounded away from $0$ uniformly in $n$. 
\item 
Building on this, in the second step, it is showed that the $\phi_p$-crossing probabilities of $2n\times n$ rectangles
in the hard direction are also bounded away from $0$ uniformly in $n$.
\item 
Finally, in the third step, assuming that $\tilde p_c < p_c$, it is showed that for $p \in (\tilde p_c, p_c)$, 
$\phi_{p} (\calC_h(2n,n)) \to 1$, as $n\to \infty$. 
The first step then implies that the dual of $\phi_{p'}$ exhibits exponential decay 
for any $p' \in (p, p_c)$, and this contradicts the fact that $p < p_c$. 
\end{itemize*} 

While the first step is not specific to planar lattices, the next two steps make use of planarity,
namely by "gluing" crossings and invoking duality. 
In Sections~\ref{sec:proof_RSW2} and~\ref{sec:mainproof} of the present paper, 
we adapt the arguments used in the last two steps to the setting of slabs.
An essential element is the "gluing" lemma discussed in Section~\ref{sec:glue}.

Adapting the final step requires particular attention,
since the dual of a random cluster measure $\phi_{p}$ on $\slab$ is not a random-cluster measure itself. 
To overcome this difficulty, we use certain bounds 
on the speed of convergence of $\phi_{p} (\calC_h(2n,n))$ to $1$ for $p \in (\tilde p_c,p_c)$.

\paragraph{Differential inequalities.}


For an event $A$ and a configuration $\om$ let $H_A(\om)$ be the Hamming distance between $\om$ and $A$,
that is the minimal number of edges whose state needs to be altered to obtain from $\om$ a configuration $\om' \in A$. 
Thus $H_A$ is a random variable taking non-negative integer values. 
Moreover, If $A$ is an increasing event, then $H_A$ is a decreasing random variable. 

The following lemma is the integrated form of the differential inequality of~\cite{GriPiz97}, as written in~\cite[Rem. 2.4]{DumMan14}. 
This is the cornerstone of our approach. 
\begin{lemma}
	Let $A$ be an increasing event depending only on the state of finitely many edges. Then, for $0< p < p' <1$,
	\begin{align}\label{eq:hamming_inf}
			\phi_{p'}(A)\ge \phi_{p}(A)\exp\big[4(p' - p)\phi_{p'}(H_A)\big],
	\end{align}
	where $\phi_{p'}(H_A)$ is the expectation of $H_A$ under $\phi_{p'}$.
	Similarly, if $A$ is decreasing,
		\begin{align}\label{eq:hamming_dec_inf}
			\phi_{p'}(A)\le \phi_{p}(A)\exp\big[-4(p'-p)\phi_{p}(H_A)\big].
	\end{align}
\end{lemma}

The following lemma taken from~\cite[Thm.~3.45]{Gri06} will also be useful. 
\begin{lemma}\label{lem:bnddistoevent}
	Let $0< p < p' <1$. For any non-empty increasing event $A$, and any non-negative integer $k$,
	\begin{equation} \label{eq:chngphambnd}
		 \phi_{p}(H_A \leq k) \le C^k \phi_{p'}(A),
	\end{equation}
	where 
	$$C= \frac{q^2(1-p)}{(p'-p)[p'+q(1-p)]}.$$
\end{lemma}

\section{Gluing Lemma} \label{sec:glue}

One of the main challenges in percolation in dimensions higher than 2 is that Jordan's theorem does not apply. As a consequence, it is difficult to connect open paths together. Indeed, contrary to planar graphs, on non-planar graphs such as slabs, paths may overlap without intersecting. 
The gluing lemma is a tool to overcome this obstacle for slabs 
or for models with finite range interactions. 
Here we will only present it in the context of slabs. 

\begin{lemma}[Gluing Lemma] \label{lem:gluing}
	Let $D$ be a subset of $\Lat$ and $A_1,A_2,B_1,B_2 \subset D$.
	Suppose that the following deterministic topological condition is satisfied:
	\begin{align}\label{eq:topo_cond}
	\text{Any two paths $\chi, \ga \subset D$ connecting $A_1$ to $A_2$ and $B_1$ to $B_2$, respectively, intersect.}
	\end{align}
	
	In addition let $D'$ be a subset of $\Lat$ containing $D$ and $A_0$ be a subset of $D'$. 
	Define $\calA$ as the event that there exists an open cluster $C \subset \ovr{D'}$ intersecting $\ovr{A_0}$
	and that contains a path $\chi \subset \ovr{D}$ connecting $\ovr{A_1}$ and $\ovr{A_2}$.
	Let $\calB$ be the event that $\ovr{B_1}$ is connected to $\ovr{B_2}$ by an open path contained in $\ovr{D}$.
	Finally let $\calX$ be the event that there exists an open cluster $C' \subset D'$ that 
	intersects $\ovr A_0$ and contains a path $\ga \subset \ovr{D}$ connecting $B_1$ to $B_2$.
	Then the two following statements hold. 
	\begin{enumerate}[label=(\roman*)]
		\item \label{lem:gluing1} There exists a constant $c > 0$, only depending on $p$, $\Lat$ and $S$,  such that
			\begin{align}\label{eq:gluing1}
				\phi_p(\calX)\geq c\:\phi_p(\calA) \phi_p(\calB).
			\end{align}
		\item \label{lem:gluing2} There exists a constant $\beta > 0$, only depending on $p$, $\Lat$ and $S$, such that
			\begin{align}\label{eq:gluing2}
				\phi_p(\calX)\geq \phi_p(\calA) \phi_p(\calB)- \left(1-\phi_p(\calA)\right)^\beta.
			\end{align}
	\end{enumerate}
\end{lemma}

\begin{figure}
\begin{center}
  \includegraphics[width=0.9\textwidth]{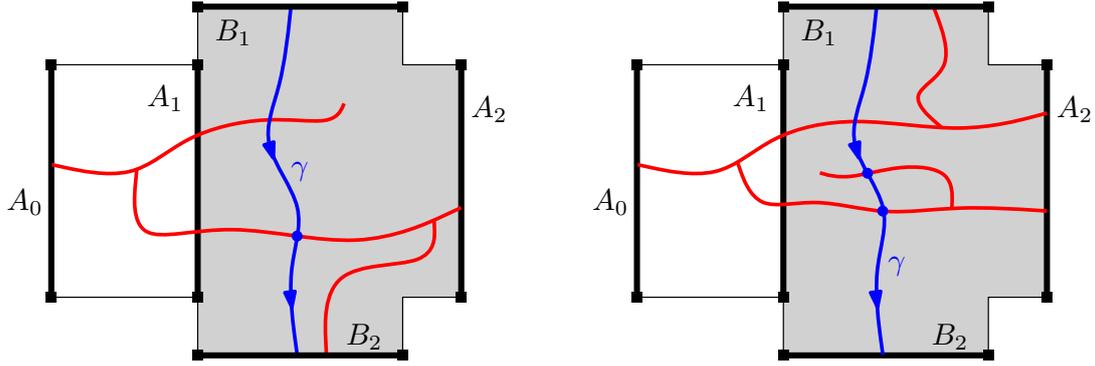}
\end{center}
\caption{The typical use of Lemma~\ref{lem:gluing} seen from "above". 
The grey area is $D$ which forms $D'$ with the additional white rectangle.
The blue path ensures the occurrence of $\calB$; the red cluster is the one in the definition of $\calA$. 
The two overlap but do not necessarily connect. \newline
Left: a configuration in $\calY^{(1)}$ but not in $\calY^{(2)}$;
right: a configuration in both $\calY^{(1)}$ and $\calY^{(2)}$. 
The overlap points are marked.}
\label{fig:gluing}
\end{figure}

The first statement may be understood as follows. 
If two open paths necessarily overlap, then they have a positive probability of being connected to each other. 
The second statement is a quantitative version of the first, useful in Section~\ref{sec:mainproof}.
It essentially states that if $\calA$ occurs with high probability, then the overlapping paths connect with high probability.  


Initially a version of this lemma appeared in~\cite{duminil2014absence} in the context of Bernoulli percolation.
Its proof does not essentially use independence; it relies on the finite-energy property, a property shared by the random-cluster model.
The property states that for any configuration $\om_0$ and any edge $e$
\begin{align}\label{eq:finiteenergy}
	\frac{p}{p + q(1-p)}\leq \phi_p\big[\om(e) = 1 \;\big|\; \om(f) = \om_0(f), \forall f \neq e\big] \leq p.
\end{align}
The second part of the lemma, although similar in spirit, requires several additional technical tricks.
We give a full proof of the two parts below. 
To help legibility, 
we start with the simpler statement~\ref{lem:gluing1},
we then discuss the additional elements needed to obtain~\ref{lem:gluing2}.

\newcommand{\preceqlex}{\preceq_{\text{lex}}}

\begin{proof}[Lemma~\ref{lem:gluing}\ref{lem:gluing1}]
Set $\calY = (\calA \cap \calB) \setminus \calX$.
In addition, for $i = 1,2$, let $\calY^{(i)} \subset \calY$ be the event that
there exists an open cluster in $\ovr{D}$ that 
contains a crossing from $A_1$ to $A_2$,
does not intersect $B_i$ and
is connected to $A_0$ in $\ovr{D'}$.
See Figure \ref{fig:gluing} for examples.
Note that $\calY^{(1)}$ and $\calY^{(2)}$ are not necessarily disjoint, but $\calY^{(1)} \cup \calY^{(2)} = \calY$.
The rest of the proof is dedicated to bounding the probability of $\calY^{(1)}$.

Let $\preceq$ be an ordering of the oriented edges of $\ovr{D}$. 
This induces a lexicographical ordering of the paths contained in $\ovr{D}$, 
which we will denote $\preceqlex$. 

For $\om \in \calB$, let $\ga = \ga(\om)$ be the minimal open path (for $\preceqlex$) 
contained in $\ovr D$, from $\ovr B_1$ to $\ovr B_2$.
We call a point $z \in \ga(\om)$ an \emph{overlap point} if there exists a cluster as in the definition of $\calY^{(1)}$ 
that intersects $\ovr{\{z\}}$.

We now define a map $\Psi : \calY^{(1)} \to \calX$ as follows. 
For $\om \in \calY^{(1)} $, because of the topological condition~\eqref{eq:topo_cond},
there exists at least one overlap point $z \in D$. 
We choose arbitrarily one such overlap point $z = z(\om)$. 

We define $\Psi(\om)$ by modifying the configuration $\om$ inside the region $\Lambda_2(z)$ as follows. 
Let $\ga_{i}$ and $\ga_j$ be the first and last points, respectively, of $\Lambda_1(z)$ visited by $\ga$. 
Let $a_0$ be a point of $ \partial \Lambda_1(z)$, connected to $A_0$ by an open path $(a_0, \dots, a_m)$, 
with $a_1, \dots, a_m \in \ovr{D'} \setminus \Lambda_1(z)$.
The existence of such a point is guaranteed by the fact that $z$ is an overlap point. 
In $\Psi(\om)$, edges with no endpoint in $\Lambda_1(z)$ have the same state as in $\om$. 
All edges with exactly one endpoint in $\Lambda_1(z)$ are declared closed, with the exception of
$(\ga_{i-1}, \ga_i)$, $(\ga_j,\ga_{j+1})$ and $(a_0,a_1)$, which are open 
(note that since $\om\notin \calX$, these three edges are distinct).
The edges with both endpoints in $\Lambda_1(z)$ are closed, with the exception of two open edge-disjoint paths 
$g = (g_0, \dots, g_k) \subset \Lambda_1(z)$ and $h = (h_0, \dots h_\ell)\subset \Lambda_1(z)$ such that
\begin{itemize*}
\item $g_0 = \ga_i$, $g_k = \ga_j$
\item $h_0 = g_t$ for some $1 \leq t \leq k-1$ and $h_\ell = a_0$ 
\item $(g_t, g_{t+1}) \preceq (g_t, h_1)$ and $g_t \in \ovr{\{z\}}$ (where $t$ is such that $g_t = h_0$).
\end{itemize*}
The existence of such a modification may easily be checked and we do not give additional details here.
See Figure~\ref{fig:connectingsurgery} for an illustration.
It is immediate that $\Psi(\om)$ is indeed in $\calX$. 

In order to compare $\phi_p(\calX)$ to $\phi_p(\calA\cap \calB)$, we will use the following simple relation
\begin{align}
\phi_p\big(\calY^{(1)}\big) 
&= \sum_{\om \in \calY^{(1)}} \phi_p( \Psi(\om) ) \frac{\phi_p(\om)}{\phi_p( \Psi(\om) )}\nonumber \\
&\leq \sup_{\om \in \calY^{(1)}}\frac{\phi_p(\om)}{\phi_p( \Psi(\om) )} 
\cdot \sup_{\sigma \in \Im(\Psi)}|\Psi^{-1}( \sigma ) |
\cdot \sum_{\sigma \in \Im(\Psi)} \phi_p( \sigma ).
\label{eq:simple_valued_map}
\end{align}
Since $\om$ and $\Psi(\om)$ only differ in $\Lambda_2(z)$, 
the finite energy property~\eqref{eq:finiteenergy} implies
$$\sup_{\om \in \calY^{(1)}}\frac{\phi_p(\om)}{\phi_p( \Psi(\om) )} \leq \left(\frac{q}{\min\{ p, 1-p\}}\right)^{|\Lambda_2|},$$
where $|\Lambda_2|$ denotes the number of edges of $\Lambda_2$.
Moreover, since $\Psi$ takes values in $\calX$, we have $\sum_{\sigma \in \Im(\Psi)} \phi_p( \sigma ) \leq \phi_p(\calX)$.

Let us now bound $\sup_{\sigma \in \Im(\Psi)}|\Psi^{-1}( \sigma ) |$.
Fix $\sigma \in \Im(\Psi)$ and $\om \in \Psi^{-1}( \sigma )$.
Recall that $\ga(\sigma)$ is the minimal $\sigma$-open path contained in $\ovr D$, from $\ovr B_1$ to $\ovr B_2$. 
(Such a path necessarily exists since $\sigma \in \calX$.)
We claim that, due to the nature of the modification applied to $\om$ in order to obtain $\Psi(\om) = \sigma$
and to the fact that $\preceqlex$ is lexicographical, 
$\ga(\sigma)$ coincides with $\ga(\om)$ 
up to the first time it enters $\Lambda_1(z)$ 
and after the last time it exits  $\Lambda_1(z)$. 
More precisely, we claim that $\ga(\sigma)$ is the concatenation of $\ga_{[0,i]}(\om)$, $(g_0, \dots, g_k)$ and $\ga_{[j,n]}(\om)$, 
(where $n$ is the length of $\ga(\om)$ and $i,j$ and $(g_0, \dots, g_k)$ are defined above). 
This fact is essential, and we give a detailed explanation below. 
\medskip

Let $\chi$ be the concatenation of $\ga_{[0,i]}(\om)$, $(g_0, \dots, g_k)$ and $\ga_{[j,n]}(\om)$ and 
suppose that $\ga(\sigma)\neq\chi$. 
Since $\chi$ is open in $\sigma$, it must be that $\ga(\sigma) \prec_{\text{lex}} \chi$. 
Let $\tau = \inf\{i \geq 0:  \chi_i \neq \ga_i(\sigma)\}$.
There are three possible situations; we analyse them separately and show that each leads to a contradiction.

Suppose $\tau > i + k$, i.e. $\ga(\sigma)$ and $\chi$ differ after exiting $\Lambda_1(z)$. 
Then $\ga(\sigma)$ can not visit $\Lambda_1(z)$ again, since the boundary of $\La_1(z)$ has only three $\si$-open incoming edges, 
and two of them have already been visited by $\ga(\si)$. 
Hence the latter part of $\ga(\sigma)$, namely $\ga_{[\tau, |\ga(\sigma)|]}(\sigma)$, is open in $\om$ as well as in $\sigma$.
Moreover $\ga_{[\tau, |\ga(\sigma)|]}(\sigma) \prec_{\text{lex}} \chi_{[\tau, |\chi|]} = \ga_{[j, |\ga(\om)|]}(\om)$, 
which contradicts the minimality of $\ga(\om)$. 

Suppose $i <\tau \leq i + k$, i.e. that $\ga(\sigma)$ and $\chi$ differ when in $\Lambda_1(z)$. 
Then the only possibility is that $\ga_{\tau}(\sigma) = h_1$ while $\chi_\tau = g_{t+1}$.
Since $(g_{t}, g_{t+1})\preceq (h_0,h_1)$, this contradicts the minimality of $\ga(\sigma)$. 

Finally suppose $\tau < i$. In particular
\begin{align*}
\ga_{[0, \tau+1]}(\sigma) \prec_{\text{lex}} \chi_{[0, \tau+1]} = \ga_{[0, \tau+1]}(\om). 
\end{align*}
The minimality of $\ga(\om)$ then implies that $\ga(\si)$ is not $\om$-open, 
hence it uses an edge with at least one endpoint in $\La_1(z)$.
This occurs after time $\tau$, since $\ga_{[0, \tau]}(\sigma) =\ga_{[0, \tau]}(\om)$ does not intersect $\La_1(z)$.

Let $\tau' > \tau$ be the first time $\ga(\sigma)$ visits the box $\Lambda_1(z)$.
Then $\ga_{\tau'}(\sigma) \in \{\ga_i(\om), \ga_j(\om), a_0\}$
(these are the only points of $\partial\La_1(z)$ accessible by $\si$-open edges from the outside).
It is not possible that $\ga_{\tau'}(\sigma) = a_0$, 
since $a_0$ is not connected to $B_1$ in $\ovr{D} \setminus \La_1(z)$ in the configuration $\sigma$
(we use the fact that  $\sigma = \om$ outside $\Lambda_1(z)$ and that $\om \in \calY^{(1)}$).
Thus $\ga_{\tau'}(\sigma) \in \ga(\om)$.
In other words, $\ga(\si)$ separates from $\ga(\om)$ at time $\tau$, 
then later joins $\ga(\om)$ again before visiting $\La_1(z)$. Let us show that this is impossible. 

Let $\tau''$ be the first time after $\tau$ when $\ga_{\tau''}(\sigma) \in \ga(\om)$.
The above discussion implies that
$\tau <\tau'' \leq \tau'$ and $\ga_{[0,\tau'']}(\sigma) \prec_{\text{lex}} \ga(\om)$.
This contradicts the minimality of $\ga(\om)$, 
since $\ga_{[0,\tau'']}(\sigma)$ is open in $\om$ 
and represents a more optimal first section for a connection from $B_1$ to $B_2$ in $D$. 

This concludes the proof of $\ga(\sigma) = \chi$. Let us return to the analysis of $\Psi^{-1}(\sigma)$. 
\medskip

\begin{figure}
\begin{center}
  \includegraphics[width=0.45\textwidth]{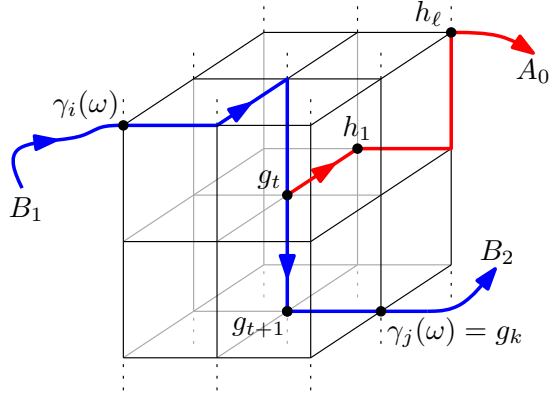}
\end{center}
\caption{The local modification performed on $\om$ in and around $\Lambda_1(z)$ to obtain $\Psi(\om)$. 
The blue path is $g$ and the red is $h$. Note that $(g_{t}, g_{t+1})\preceq (h_0,h_1)$. 
The central axis in the image is $\ovr z$. }
\label{fig:connectingsurgery}
\end{figure}

Note that $g_t$ is the unique point $x \in \ga(\sigma)$ that is connected by a $\sigma$-open path to $A_0$ in $D' \setminus \ga(\sigma)$.
Thus $g_t$ is determined by $\sigma$, and so is $z$, the first coordinate of $g_t$. 
Since $\om$ and $\sigma$ differ only inside $\Lambda_2(z)$, we obtain the bound
$$|\Psi^{-1}( \sigma ) | \leq 2^{|\Lambda_2(z)|} \quad \text{ for all } \sigma \in \Im(\Psi).$$
It follows from~\eqref{eq:simple_valued_map} and the above bounds that
\begin{align*}
	\phi_p(\calY^{(1)}) 
	\leq \left(\frac{2q}{\min\{ p, 1-p\}}\right)^{|\Lambda_2|} \phi_p( \calX).
\end{align*}
The same bound applies to $\phi_p(\calY^{(2)})$, and combining the two yields
\begin{align*}
	\phi_p(\calA) \phi_p( \calB) - \phi_p(\calX) 
	\leq \phi_p(\calY)
	\leq \phi_p(\calY^{(1)}) + \phi_p(\calY^{(2)}) 
	\leq 2\left(\frac{2q}{\min\{ p, 1-p\}}\right)^{|\Lambda_2|}  \phi_p( \calX),
\end{align*}
which leads to~\eqref{eq:gluing1}. 
\end{proof}

The idea for the proof of the second statement is that, if $\calA$ has high probability, 
then typically there must be a large number of overlap points, 
otherwise the connection between $A_0$, $A_1$ and $A_2$ could easily be broken (this is proved in Lemma~\ref{lem:Ysmall}).
Using this fact, we may associate to a configuration $\om \in (\calA\cap \calB) \setminus \calX$
not one, but many configurations $\sigma \in \calX$.
This in turn implies, using Lemma~\ref{lem:multivaluemap} below, 
that $\calX$ has much higher probability than $(\calA\cap \calB) \setminus \calX$.

Several technical difficulties occur in this argument, and the proof requires some new ingredients. 
In particular, the ordering of the edges used for defining the minimal path $\ga(\om)$ needs to be random. 

Let $\calO$ denote the set of total orderings of oriented edges of $\ovr{D'}$ and $\mu$ be the uniform measure on $\calO$.
Set $\nu = \phi_p \otimes \mu$ to be the measure on $\{0,1\}^{E(\slab)} \times \calO$ obtained as the product of $\phi_p$ and $\mu$.

\begin{lemma}\label{lem:multivaluemap}
Suppose we have two event $\calE, \calF \subseteq \{0,1\}^{E(\slab)} \times \calO$, and a map $\Psi$ from $\calE$ to $2^\calF$. 
Suppose that the following statements are true:
\begin{enumerate}
	\item
		If $(\omega, \preceq) \in \calE$ and $(\omega', \preceq')  \in \Psi(\omega, \preceq)$, 
		then $\preceq = \preceq'$.
	\item
		There exists $t > 0$ such that for each $(\omega, \preceq) \in \calE$, we have $\vert\Phi(\omega, \preceq)\vert \geq t$. 
	\item
		There exists $s$ such that for each $(\omega', \preceq) \in \calF$ 
		there exists a finite set $S(\om', \preceq)$ of edges with $\vert S(\om', \preceq) \vert \leq s$, 
		such that the configurations in $\Psi^{-1}(\omega', \preceq) := \{ \omega : (\omega',\preceq) \in \Psi(\omega, \preceq)\}$ 
		differ from $\omega'$ only inside $S(\om',\preceq)$.
	\end{enumerate} 
	Then the following statement is true:
	\begin{equation}
	\label{eq:multivaluemap}
	\nu(\calE)\leq \frac{1}{t}\Big(\frac{2q}{\min\{p,1-p\}}\Big)^s \nu(\calF).
	\end{equation}
\end{lemma}

\begin{proof}
	The lemma is a generalization of~\cite[Lem. 7]{duminil2014absence}, and the proof is similar.
	Due to the finite energy property~\eqref{eq:finiteenergy} and to the third condition, 
	for all $(\om, \preceq) \in \calE$ and $(\om', \preceq) \in \Psi(\om, \preceq)$, we have
	$$
		\nu(\om, \preceq) 
		= \mu(\preceq) \phi_p(\om) 
		\leq \left(\frac{q}{\min\{ p, 1-p\}}\right)^s\cdot \mu(\preceq) \phi_p(\om')
		= \left(\frac{q}{\min\{ p, 1-p\}}\right)^s\cdot \nu(\om', \preceq).
	$$
	By summing over $(\om, \preceq) \in \calE$ and $(\om', \preceq) \in \Psi(\om, \preceq)$, we obtain: 
	\begin{align*}
		\nu(\calE)  
		&\leq \frac{1}{t} \left(\frac{q}{\min\{ p, 1-p\}}\right)^s 
			\sum_{(\om,\preceq) \in \calE} \sum_{(\om',\preceq) \in \Psi(\om, \preceq)} \nu(\om', \preceq) \\
		&= \frac{1}{t} \left(\frac{q}{\min\{ p, 1-p\}}\right)^s 
			\sum_{(\om',\preceq) \in \calF}  | \Psi^{-1}(\om', \preceq)| \cdot \nu(\om', \preceq) \\
		&\leq \ \frac{1}{t} \left(\frac{q}{\min\{ p, 1-p\}}\right)^s 	
			\sum_{(\om',\preceq) \in \calF} 2^s \: \nu(\om', \preceq)
		\ = \ \frac{1}{t} \left(\frac{2q}{\min\{ p, 1-p\}}\right)^s \nu(\calF).
	\end{align*}
\end{proof}

\begin{proof}[Lemma~\ref{lem:gluing}\ref{lem:gluing2}]
	Let $\om \in \calB$ and $\preceq \: \: \in \calO$. 
	As before, let $\preceqlex$ be the lexicographical order induced by $\preceq$ on oriented self-avoiding paths of $\ovr{D'}$.
	Set $\ga^{(1)} = \ga^{(1)}(\om, \preceq)$ to be the $\preceqlex$-minimal open path of $\ovr D$ from $\ovr B_1$ to $\ovr B_2$, 
	and $\ga^{(2)}(\om, \preceq)$ the $\preceqlex$-minimal open path of $\ovr D$ from $\ovr B_2$ to $\ovr B_1$.
	
	As in the previous proof, set $\calY = (\calA \cap \calB) \setminus \calX$ and consider $\om \in \calY$. 
	In the previous proof, we have defined overlap points. 
	Since in the present proof we will need to work with $\ga^{(1)}$ and $\ga^{(2)}$ simultaneously,
	we will define $(1)$-overlap points and $(2)$-overlap points.
	For $i =1,2$, let $W^{(i)} = W^{(i)}(\om, \preceq)$ be the set of points $z \in \Lat$, 
	such that $\ovr{\{z\}}$ intersects $\ga^{(i)}$ and 
	also intersects an open cluster $C$ of $\ovr D$ with the following properties: 
	\begin{itemize*}
	\item $C$ contains a path from $\ovr{A_1}$ to $\ovr{A_2}$,
	\item $C$ does not intersect $\ovr{B_i}$, 
	\item $C$ is connected to $\ovr{A_0}$ in $\ovr{D'}$.
	\end{itemize*}
	Call the points of $W^{(i)}$ $(i)$-overlap points. 
	Obviously a point can be simultaneously both a $(1)$ and $(2)$-overlap point.  
	See Figure~\ref{fig:gluing2} for an illustration. 
	
	Since $\om \in \calY$, any crossing in $\ovr D$ between $A_1$ and $A_2$ as in the definition of $\calA$ 
	necessarily contains at least one overlap point of each type.
		
	We also introduce the following related notion.
	For $i =1,2$, we say a point $z \in D$ is an \emph{(i)-almost-overlap point} if
	there exists $z' \in \La_1(z)$ and $s,s' \in S$ such that
	\begin{itemize*}
	\item $(z,s) \in \ga^{(i)}$,
	\item $(z',s')$ is connected to $\ovr A_0$ in $\ovr{D' \setminus \{z\}}$,
	\item $(z',s')$ is not connected to $\ovr B_i$ in $\ovr{D}$,
	\item $(z,s') \notin \ga^{(i)}$. 
	\end{itemize*}
	Let $U^{(i)}(\om, \preceq)$ denote the set of $(i)$-almost-overlap points.
	It will be useful to note that for $i=1,2$, $W^{(i)}(\om, \preceq) \subset U^{(i)}(\om, \preceq)$. 
	To be precise, an $(i)$-almost-overlap point is a $(i)$-overlap point if in addition to the conditions above, 
	$z = z'$ and $(z,s')$ is connected to both $A_1$ and $A_2$ in $\ovr D$. 
	
	Our aim is to bound $\phi_p(\calY) = \nu(\calY \times \calO)$. 
	To do this we will split $\calY$ in three events. 
	Since these will depend on the (random) order $\preceq$, we will henceforth work with couples $(\om,\preceq)$. 
	
	Fix a constant $c >0$ that we will identify later in the proof (see the end of the proof of Lemma~\ref{lem:Ysmall}),
	and define $\alpha =  -c  \log (\phi_p(\calA^c))$. 
	Define the following events:
	\begin{align*}
		\calY_{\leq \alpha} &= 
		\Big\{ (\om, \preceq):\  \om \in \calY, 
		\vert U^{(1)}(\om, \preceq) \vert \leq \alpha \text{ and } \vert U^{(2)}(\om, \preceq) \vert \leq \alpha \Big\}, \\
		\calY_{> \alpha}^{(1)} &= \big\{(\om, \preceq):\  \om \in \calY,\vert U^{(1)}(\om, \preceq) \vert > \alpha \big\},\\
		\calY_{> \alpha}^{(2)} &= \big\{(\om, \preceq):\  \om \in \calY, \vert U^{(2)}(\om, \preceq) \vert > \alpha \big\}.
	\end{align*}
	Note that indeed $\calY\times \calO = \calY_{\leq \alpha} \cup \calY_{> \alpha}^{(1)} \cup \calY_{> \alpha}^{(2)}$, 
	but that the two latter events are not necessarily disjoint. 
	We start by bounding the probability of the first event.

	\begin{figure}
	      \begin{center}
	        \includegraphics[width=0.7\textwidth]{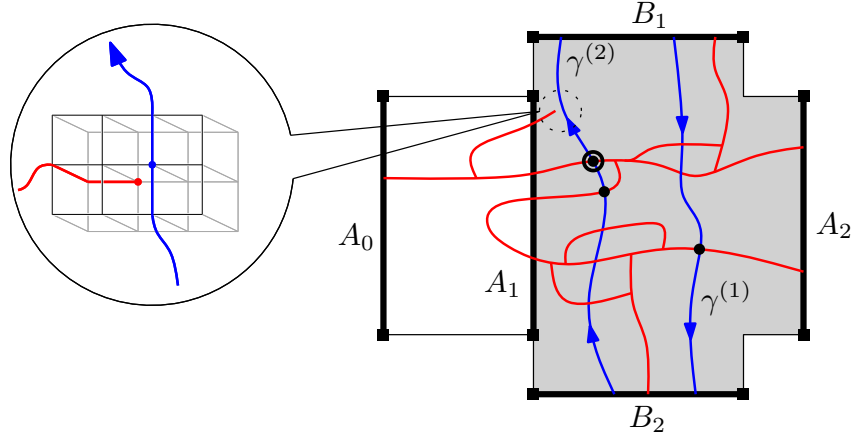}
	      \end{center}
	      \caption{The red cluster ensures the occurrence of the event $\calA$. 
	      The blue paths are $\ga^{(1)}$ and $\ga^{(2)}$.
	      Not all intersections between the blue and red paths are overlap points, 
	      only the marked ones are. 
	      Out of these, only the doubly marked point is in $\calW$. 
	      The encircled region contains a $(2)$-almost-overlap point. }
	      \label{fig:gluing2}
	\end{figure}
	
	\begin{lemma}\label{lem:Ysmall} 
		Provided that the constant $c > 0$ in the definition of $\alpha$ is small enough, we have
		$$\nu \big( \calY_{\leq \alpha} \big) \leq \sqrt{\phi_p(\calA^c)}.$$
	\end{lemma}
	
	The idea behind this lemma is that, for $(\om, \preceq) \in \calY_{\leq \alpha} $,
	the connection between $A_1$ and $A_2$ in $\ovr D$ is fragile, 
	since it only has few overlap points with $\ga^{(1)}$ and $\ga^{(2)}$. 
	Thus, it is easy to break this connection, 
	and this leads to an upper bound on $\nu(\calY_{\leq \alpha})$ in terms of $ \phi_p(\calA^c)$.
	
	\begin{proof}
		Define a map $\Psi: \calY_{\leq \alpha}\rightarrow \calA^{c} \times \calO$ as follows.  
		Take $(\om,\preceq) \in \calY_{\leq \alpha}$. 
		Let $\calW(\om, \preceq)$ be the set of points $z \in W^{(1)}\cup W^{(2)}$ 
		such that $\ovr{\{z\}}$ is connected to $\ovr{A_0}$ without 
		using other points of $\ovr{W^{(1)}\cup W^{(2)}}$.
		It is essential to remark that, since $\om \in \calY$, $\calW(\om, \preceq) \neq \emptyset$. 
		
		Let $\Psi(\om, \preceq) = (\si,\preceq)$ 
		with $\si$ equal to $\om$ for all edges with no end-point in $\ovr{\calW(\om, \preceq)}$.
		The edges with at least one end-point in $\ovr{\calW(\om,\preceq)}$ are declared closed in $\si$, 
		unless they are part of $\ga^{(1)}$ or $\ga^{(2)}$, in which case they remain open. 
		
		Let us show that $\Psi(\om,\preceq) \in \calA^c \times \calO$, that is $\si \notin \calA$.
		Suppose that this is not the case, and that $\si \in \calA$.
		We know that $\si \in \calB$ ($B_1$ and $B_2$ being united by $\ga^{(1)}$ and $\ga^{(2)}$);
		moreover $\si \leq \om$ so $\si\notin \calX$. 
		We therefore conclude that $\si \in \calY$.
		Then, by the topological condition~\eqref{eq:topo_cond}, in $\si$
		there exists at least one overlap point $z_0$, which is connected to $A_0$ in $\ovr{D'}$.
		Let $\chi$ be a $\si$-open path in  $\ovr{D'}$ from $z_0$ to $A_0$. 
		Denote $(z_1, s_1)$ the last point on $\chi$ such that $z_1$ is an overlap point. 
		Then $z_1 \in \ovr{\calW(\om,\preceq)}$ and hence the edges emanating from $(z_1,s_1)$ should be closed in $\si$. 
		This contradicts the fact that $(z_1,s_1)$ is connected to $A_0$ in $\si$.
		We have therefore shown that $\si \notin \calA$.
		
		We now use Lemma~\ref{lem:multivaluemap} to bound the probability of the event under study.
		Condition 1 is satisfied by definition; condition 2 is satisfied with $t = 1$. 
		We focus on the third condition. 
		
		Let $(\sigma,\preceq) \in \Im(\Psi)$ and $\om \in \Psi^{-1}(\sigma,\preceq)$. 
		Since any open edge of $\sigma$ is also open in $\om$ and $\ga^{(1)}(\om)$ and $\ga^{(2)}(\om)$
		are both open in $\sigma$, we have $\ga^{(i)}(\sigma) = \ga^{(i)}(\om)$ for $i=1,2$. 
		Moreover, in going from $\om$ to $\sigma$, we do not create new almost-overlap points, 
		i.e. $U^{(i)}(\sigma,\preceq) \subset U^{(i)}(\om,\preceq)$.
		Finally we observe that, by definition of $\Psi$, 
		all points where modifications were made when going from $\om$ to $\sigma$, 
		are almost-overlap points of $\sigma$.
		In conclusion, $\om$ and $\sigma$ only differ in the vicinity of points in 
		$U^{(1)}(\sigma,\preceq) \cup U^{(2)}(\sigma,\preceq)$, 
		and there are at most $2 \alpha$ of these. 
		It follows that the third condition of Lemma \ref{lem:multivaluemap} is satisfied with
		$ s = 2 \alpha K$,  where $K$ is the number of edges of $\La_1$. 

		Using the definition of $\alpha$ and Lemma \ref{lem:multivaluemap}, we obtain
		\begin{align*}
			\nu(\calY_{\leq\alpha})
			\leq\left(\frac{2q}{\min\{p,1-p\}}\right)^{2\alpha K} \nu(\calA^{c} \times \calO)
			= \phi_p(\calA^c)^{1 -  2 c K \log\left(\frac{2q}{\min\{ p, 1-p\}}\right) },
		\end{align*}
		which implies the lemma provided that
		$c \leq \left[ 4 K \log\left(\frac{2q}{\min\{ p, 1-p\}}\right)\right]^{-1}$.
	\end{proof}
	
	\begin{remark}
		Given a configuration $\si$ in the image of $\Psi$, the almost-overlap points of $\si$, 
		rather than simply the overlap points, 
		are the places where modifications may have been performed when constructing $\si$ from one of its pre-images.
		This explains the necessity of introducing the additional notion of almost-overlap point.
	\end{remark}
	
	We will now focus on bounding the probabilities of $\calY_{> \alpha}^{(i)}$ for $i =1,2$. 
	More specifically we will prove the following.
	
	\begin{lemma}\label{lem:Ylarge}
		There exists a constant $\beta >0$ depending on $p$, $\Lat$ and $S$ only, such that, for $\phi_p(\calA^c)$ small enough  
		\begin{align}\label{eq:Ylarge}
			\nu \big(\calY^{(i)}_{\geq \alpha}  \big) \leq \phi_p(\calA^c)^\beta,
		\end{align}
		for $i=1,2$. 
	\end{lemma}
	
	By symmetry we can concentrate on bounding $\phi_p(\calY_{> \alpha}^{(1)})$. 
	To simplify notation, we will henceforth omit the index $(1)$. 
	
	The idea behind this lemma is that, for $(\om, \preceq) \in \calY_{> \alpha}$,
	the multitude of almost-overlap points gives many opportunities for $\ga$ to connect to $A_0$. 
	Thus, the probability of $\calY_{> \alpha}$ should be much smaller than that of $\calX$, 
	and this will ultimately yield the bound \eqref{eq:Ylarge}.
	
	To make this heuristic rigorous, we will define a multi-valued map $\Psi: \calY_{> \alpha}\to 2^{\calX \times \calO}$ 
	and apply Lemma \ref{lem:multivaluemap}.
	As suggested above, the function $\Psi$ will consist in connecting $A_0$ to $\ga$ by modifying the configuration locally 
	around certain almost-overlap points; 
	we say we will perform a connecting surgery at these points. 
	Not all almost-overlap points are suited to perform the connecting surgery, 
	and we start by identifying those who are. 

	Fix in $S$ an arbitrary system of geodesics uniting any pair of points $s, s' \in S$;
	such a system always exists since $S$ is connected. 
	We may then talk of the segment between $s$ and $s'$, which we denote by $[s,s']$.
	As mentioned in the introduction, one may think of $S =  \{0,\dots,k\}$, in which case the segment between $s$ and $s' > s$ is simply 
	$[s,s'] = \{s, s+1, \dots, s'\}$. 
	For $(\om, \preceq) \in \calY\times \calO$, 
	we call a point $z \in D$ a \emph{good almost-overlap} point, if it is an almost-overlap point 
	(with $z', s$ and $s'$ as in the definition of $(1)$-almost-overlap points) 
	and in addition
	\begin{itemize*}
	\item there is no $t$ strictly between $s$ and $s'$ such that $(z,t) \in \ga$ and
	\item if $\ga_j = (z,s)$ and if $t \in S$ is the first point after $s$ when going from $s$ to $s'$ along $[s,s']$,
		then $(\ga_j, \ga_{j+1}) \preceq ((z,s),(z,t))$.
	\end{itemize*}
	Let $V(\om, \preceq)$ be the set of good almost-overlap points. 
	The following lemma states that generally a positive proportion of almost-overlap points are good. 
	
	\begin{lemma} \label{lem:supnicehigh}
		For any configuration $\om_0 \in \calY$ and path $\gamma_0$, 
		\begin{align}\label{eq:supnicehigh}
			\nu \Big[ \vert V(\om_0,\preceq) \vert \geq \tfrac{1}{4} \vert U(\om_0, \preceq)\vert \:\Big\vert\: 
			 \om= \om_0 ; \ga(\om_0, \preceq) = \gamma_0 \Big] 
			\geq \frac{1}{4},
		\end{align}
		whenever the conditioning is not void. 
	\end{lemma}
	
   \begin{remark}
	 	It is for the above lemma alone that the random ordering is necessary. 
		Indeed, for a fixed ordering, there is no guarantee that enough good almost-overlap points exists. 
	\end{remark}
	
	\begin{proof}
		Before we start the proof, let us mention that
		the set of almost-overlap points $U(\om, \preceq)$ only depends on 
		$\om$ and $\gamma(\om, \preceq)$, not otherwise on $\preceq$. 
		The set of good almost-overlap points does however depend further on $\preceq$. 
		
		Fix $\om_0 \in \calY$ and a path $\gamma_0$.
		Let $U_0$ be the set $U(\om_0, \preceq)$ for an ordering $\preceq$ such that $\ga(\om_0, \preceq) =  \gamma_0$. 
		(Such an ordering exists if the conditioning in~\eqref{eq:supnicehigh} is not degenerate.)
		We will prove that, for each $z \in U_0$, 
		\begin{align}\label{eq:oneptinV}
			\nu \big[ z \in  V(\om,\preceq) \: \big| \: \om = \om_0, \ga(\om, \preceq) = \gamma_0 \big] 
			= \mu \big[ z \in  V(\om_0,\preceq) \: \big| \: \ga(\om_0, \preceq) = \gamma_0 \big] 
			\geq \frac{1}{2}.
	    \end{align}
	    In other words, when averaging over the choice of the order $\preceq$, any almost-overlap points is good with probability at least $1/2$. 
		This implies~\eqref{eq:supnicehigh} through a direct application of Markov's inequality.
		
		Fix $z \in U_0$ as above. Let $z' \in \Lambda_1(z)$ and $s, s' \in S$ closest to each other, 
		as in the definition of almost-overlap point, i.e. with
		\begin{itemize*}
		\item $(z,s) \in \ga_i$,
		\item $(z,s') \notin \ga$,
		\item $(z',s')$ is connected to $\ovr A_0$ in $\ovr{D' \setminus \{z\}}$,
		\item $(z',s')$ is not connected to $\ovr B_1$ in $\ovr{D}$.
		\end{itemize*}
		Let $t\in S$ be the first point after $s$ when going from $s$ to $s'$ along $[s,s']$ and let $f = ((z,s),(z,t))$. 
		Let $e = (\ga_i, \ga_{i+1})$ 
		and $e_1, \dots, e_k$ be the oriented edges emanating from $\ga_i$, other than $e$, 
		such that there exists an $\om$-open path in $D \setminus \ga_{[0,i]}$ from $(z,t)$ to $B_2$ starting with $e_i$.
		
		Under $\mu ( . \; | \; \ga(\om_0, \preceq) = \gamma_0) $ the ordering of the oriented edges emanating from $\ga_i$
		is uniform among orderings such that $e \preceq e_i$ for $i = 1, \dots, k$. 
		With this in mind, we notice that:
		\begin{itemize*}	
		\item if $f \in \{e_1, \dots, e_k\}$, then $\mu ( e \preceq f \; | \; \ga(\om_0, \preceq) = \gamma_0) =1$.
		\item if $f \notin \{e_1, \dots, e_k\}$, then $\mu ( e \preceq f \; | \; \ga(\om_0, \preceq) = \gamma_0) = \frac{k+1}{k+2}$.
		\end{itemize*}
		Equation~\eqref{eq:oneptinV} follows from the above. 
	\end{proof}
	
	\begin{proof}[Lemma~\ref{lem:Ylarge}]
		Let $\calY'_{> \alpha} = \calY_{> \alpha} \cap \big\{ \vert V(\om, \preceq) \vert  > \alpha/4 \big\}$. 
		By Lemma~\ref{lem:supnicehigh}
		$\nu(\calY'_{> \alpha}) \geq \frac12 \nu(\calY_{> \alpha}) $, and we will focus on bounding $\nu(\calY'_{> \alpha})$.
		In order to do this we will define a map 
		$\Psi: \calY'_{>\alpha} \rightarrow 2^{\calX \times {\calO}}$ 
		and apply Lemma~\ref{lem:multivaluemap} to it. 
		
		Fix a real number $c' \in(0, 1/2)$ which we will identify later and let $j=\lceil c' \alpha \rceil$. 
		Consider a pair $(\om, \preceq) \in \calY'_{>\alpha}$.
		For $z_1, \dots, z_j \in  V(\om, \preceq)$, we define $\om_{z_1, \dots, z_j}$ as follows.
		
		By definition of $V$, for each $z_k$ there exists a pair of distinct points $s_k, s_k' \in S$ 
		and a point $z_k'$ such that
		\begin{enumerate}[label=(\alph*)]\itemsep=0pt
		\item\label{cond1} $(z_k,s_k) \in \ga$,
		\item\label{cond2} $(z_k',s_k')$ is connected to $\ovr A_0$ in $\ovr{D' \setminus \{z_k\}}$,
		\item\label{cond3} $\ga$ does not intersect $\{z_k\} \times (s_k, s_k']$, 
		\item\label{cond4} $\La_1(z_k) \times (s_k, s_k')$ is not connected to $\ovr A_0$ in $\ovr{D' \setminus \{z_k\}}$,
		\item\label{cond5} if $\ga_i = (z_k,s_k)$ and $t_k$ is the first point of $S$ when going from $s_k$ to $s'_k$ 
			along $[s_k,s'_k]$, then $(\ga_i, \ga_{i+1}) \preceq ((z_k,s),(z_k,t))$.
		\end{enumerate}
		Note that conditions~\ref{cond1},\ref{cond2},\ref{cond3} and~\ref{cond5} are exactly those of the definition of $V$. 
		Condition~\ref{cond4} may be assumed by taking $s_k'$ as close to $s_k$ as possible.

		\begin{figure}
			\begin{center}
				\includegraphics[width=1\textwidth]{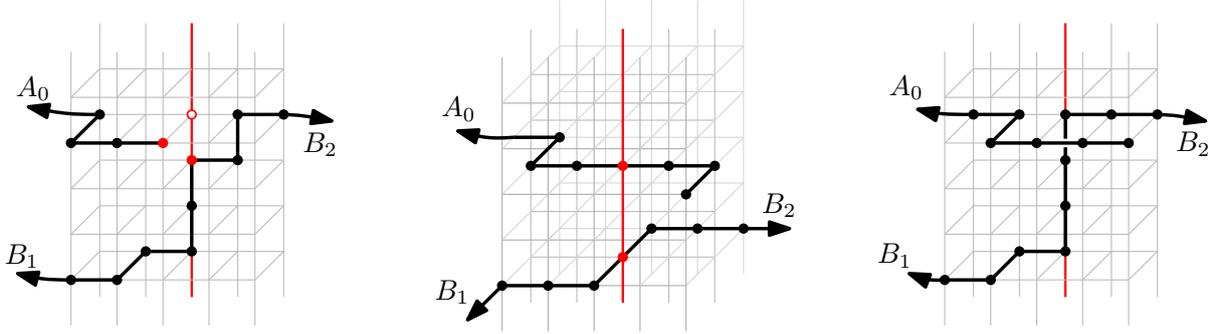}
			\end{center}
			\label{fig:overlappts}
			\caption{Left diagram: the marked vertical line corresponds to an almost-overlap point, but not an overlap point.
				The points $(z,s)$ and $(z',s')$ are marked in red and the point $(z,s')$ is marked by a circle. 
				Middle diagram: the red vertical line corresponds to an overlap point. 
				Right diagram: the red line does not correspond to an almost-overlap point, 
				since the two paths come close to each other at the same horizontal level. }
		\end{figure}
		    
		We choose points $s_k, s'_k$ and $t_k$ as above, following some deterministic ordering when several choices are possible. 
		Then $\om_{z_1, \dots, z_j}$ is identical to $\om$ except for the following edges, for each $k$:
		\begin{itemize*}
			\item all edges of $\{z_k\} \times [s_k, s_k']$ are open,
			\item if $z_{k} \neq z_k'$, then $((z_k, s_k'), (z_k',s_k'))$ is open,
			\item all edges of the form $((z_k,t),(z',t))$ with $t \in (s_k, s_k')$ are closed. 
		\end{itemize*}
		We say we obtain $\om_{z_1, \dots, z_j}$ from $\om$ by performing a connecting surgery at each point $z_k$. 
		Observe that, for any choice of $z_1, \dots, z_j \in V(\om,\preceq)$,  
		$\om_{z_1, \dots, z_j} \in \calX$.
		Indeed the connecting surgery does not close the path $\ga(\om, \preceq)$, 
		so $\om_{z_1, \dots, z_j} \in \calB$, 
		and in addition any one connecting surgery ensures that $\ga(\om, \preceq) \xlra{D'} \ovr{A_0}$.
		We set $\Psi(\om, \preceq) = \big\{(\om_{z_1,\dots, z_j}, \preceq): z_1,\dots, z_j \in V(\om,\prec) \big\}$.
		
		Let us now verify the conditions of Lemma~\ref{lem:multivaluemap}. 
		The first condition of the lemma is satisfied by definition.
		Since $|V(\om, \preceq)|  > \alpha/4$ for all $(\om, \prec) \in \calY'_{> \alpha}$,
		$|\Psi(\om, \prec)| > \binom{\alpha/4}{j}$.
		Thus the second condition is satisfied with $t = \binom{\alpha/4}{j}$
		
		Let us now study the third condition.
		The first thing to notice is that, for $(\om, \preceq)$ as above and $z_1, \ldots, z_j \in V(\om, \preceq)$,
		we have $\ga(\om, \preceq) = \ga(\om_{z_1,\ldots, z_j}, \preceq)$. 
		The proof of this fact follows the same line as the corresponding step in the proof of Lemma~\ref{lem:gluing}\ref{lem:gluing1}.
		Let us only mention that the connection surgery used here is such that $\ga(\om, \preceq)$ is open in $\om_{z_1,\ldots, z_j}$.
		Moreover, $z_1, \ldots, z_j$ were chosen as \emph{good} almost-overlap points, 
		so that the path $\ga(\om, \preceq)$ is the minimal continuation of a crossing from $B_1$ to $B_2$ at every point $z_k$. 
		Indeed, if $\ga(\om, \preceq)_j = (z_k,s_k)$, then there are only three $\om_{z_1,\ldots, z_j}$-open edges emanating from 
		$(z_k,s_k)$ and $(\ga(\om, \preceq)_j,\ga(\om, \preceq)_{j+1})$ 
		is preferable to the first edge in the link between $(z_k,s_k)$ and $(z_k',s_k')$.

		Fix now $(\si, \preceq) = (\om_{z_1,\ldots, z_j}, \preceq) \in \Psi(\om)$ 
		for some $(\om,\preceq) \in \calY_{>\alpha}$ and $z_1, \ldots, z_j \in V(\om, \preceq)$.
		Then $((z_1, s_1), \dots, (z_j, s_j))$ are the only points $(z, s)$ on $\ga(\si, \preceq)$ that are connected to 
		$\ovr{A_0}$ by a $\si$-open path only intersecting $\ga(\si, \preceq)$ at $(z,s)$. 
		Thus, $\om$ and $\si$ only differ on the set $S(\si) = \cup_{k =1}^j \Lambda_1(z_k)$.
		We insist that, since $(z_1, s_1), \dots, (z_j, s_j)$ are determined by $\si$, so is the set $S(\si)$.
		Thus the third condition of Lemma~\ref{lem:multivaluemap} is satisfied, with $s = j {|\Lambda_1|}$.
		The lemma then implies
		\begin{align*}
			\nu \left(\calY'_{>\alpha}\right) 
			\leq \frac{1}{{\alpha/4 \choose  j}}\left(\frac{2q}{\min\{p,1-p\}} \right)^{j {|\Lambda_1|}} 
			\cdot \nu\left( \calX \times \calO \right).
		\end{align*} 
		Let $Q= (\frac{2q}{\min\{p,1-p\}})^{{|\Lambda_1|}}$. 
		Note that $Q$ is a constant depending on $p, \Lat$ and $S$ only 
		and recall that $j$ was chosen as $j=\lceil c' \alpha \rceil$. 
		Since $\nu( \calX \times \calO)\leq 1$, using Stirling's formula we obtain
		\begin{equation} \label{eq:highoverlapglue}
			\nu \left(\calY'_{>\alpha} \right)  
			\leq \frac{Q^{j}}{{\alpha/4 \choose  \lceil c' \alpha \rceil}}
			\leq \big[Q \cdot (2c')^{c'} \big]^{c'\alpha} \cdot (1-2c')^{(\frac12 - c')\alpha},
		\end{equation} 
		for $\alpha$ large enough. 
		By choosing $c' \in(0,1/2)$ such that $(2c')^{c'} \leq Q^{-1}$ and setting $c_1 = (1-2c')^{(\frac12 - c')} \in (0,1)$, 
		we deduce
		$$ \nu \left(\calY'_{>\alpha} \right) \leq c_1^\alpha,$$
		for $\alpha$ large enough.
		In order to obtain the conclusion of Lemma~\ref{lem:Ylarge}, 
		recall that $\alpha =  -c  \log (\phi_p(\calA^c))$ for some constant $c$ depending on $p, \Lat$ and $S$ only, 
		and that $\alpha$ may be considered large since we restrict ourselves to small values of $\phi_p(\calA^c)$. 
	\end{proof}
	
	Let us now conclude the proof of Proposition~\ref{lem:gluing}\ref{lem:gluing2}.
	Note that the sought bound is only relevant when $\phi_p(\calA^c)$ is small.
	We will therefore prove the bound assuming $\phi_p(\calA^c)$ is small enough for Lemma~\ref{lem:Ylarge} to hold. 
	The result may be extended to any value of $\phi_p(\calA^c)$, with a possibly altered constant $\beta$. 
	
	Recall that $\calY \times \calO = \calY_{\leq \alpha} \cup  \calY^{(1)}_{> \alpha} \cup \calY^{(2)}_{> \alpha}$.
	Lemmas~\ref{lem:Ysmall} and~\ref{lem:Ylarge} bound the $\nu$-probability of the three events on the right hand side;
	we can combine them to obtain: 
	\begin{align*}
		\phi_p(\calY) = \nu(\calY) 
		\leq \nu(\calY_{\leq \alpha})
			 + \nu(\calY_{> \alpha}^{(1)})  + \nu(\calY_{> \alpha}^{(2)})
		\leq 3 \phi_p(\calA^c)^{\max\{\beta, 1/2\}}.
	\end{align*}
	The above yields~\eqref{eq:gluing2} through basic algebra. 
\end{proof}

\section{Bounds for crossing probabilities}  \label{sec:proof_RSW2}

As mentioned in the introduction, the  first step in the argument of~\cite{DumMan14} applies to non-planar graphs.
We state it here without proof:

\begin{proposition}[\cite{DumMan14} Prop 3.1]
	For $p>\tilde{p_c}$, 
	\begin{align*}
	\liminf_{n\rightarrow \infty}\phi_{p}(\calC_v(2n,n))>0.
	\end{align*}
\end{proposition} 

The object of this section is the following result, which corresponds to the second step in~\cite{DumMan14}. 
It may be understood as a Russo-Seymour-Welsh type result, with the remark that it requires increasing the value of the edge-weight.  

\begin{proposition}\label{prop:RSW}
	If $p\in(0,1)$ is such that 
	\begin{align}\label{eq:RSW}
		\liminf_{n\rightarrow \infty}\phi_{p}(\calC_v(2n,n))>0,
	\end{align}
	then for any $p'>p$,
	\begin{align}
		\liminf_{n\rightarrow \infty}\phi_{p'}(\calC_h(2n,n))>0.
	\end{align}
\end{proposition}

An immediate consequence of the two above statements is the main result of this section:

\begin{corollary}\label{cor:hardcrossbound}
	For $p>\tilde{p_c}$, 
	\begin{align*}
	\liminf_{n\rightarrow \infty}\phi_{p}(\calC_h(2n,n))>0.
	\end{align*}
\end{corollary}

Let us now focus on the proof of Proposition~\ref{prop:RSW},
the core of which lies in the following lemma. 

\begin{lemma}\label{lem:manycrossor}
	Let $ 0 < p_1 < p_2 < p_3 < 1$, and suppose that 
	\begin{align*}
		\inf \big\{\phi_{p_1}(\calC_v(2n,n)) : n \in \bbN \big\}> \delta > 0.
	\end{align*}
	There exist constants $c_0, c_1 >0$, depending only on $p_1, p_2$ and $p_3$, such that
	if $n,I \in \bbN$ are such that $1 \leq I \leq n/400$ and 
	\begin{align*}
		I^2 { \big[\phi_{p_{3}}(\calC_h(2n,n))\big]^{c_1 / I}} \leq c_0\delta,
	\end{align*}
	then 
	\begin{align}\label{eq:bndhamhor}
		\phi_{p_2}(H_{\calC_h(2n,n/2)}) \geq \frac{2^I-1}{2} \delta.
	\end{align}
\end{lemma}

In~\cite{DumMan14} it was shown that a similar statement implies Proposition~\ref{prop:RSW} (with slightly different formulations). 
This step adapts readily to the present context, and we do not give more details here; 
the interested reader is referred to~\cite[Proof of Prop.~4.1]{DumMan14}.
The rest of the section is dedicated to proving Lemma~\ref{lem:manycrossor}. 
We start with some notation. 

Let $A_1, \dots, A_K$ be subsets of vertices of some rectangle $R$ of $\slab$.
For a configuration $\om \in \{0,1\}^\slab$, 
we say that the subsets are \emph{separated} in $R$
if $A_i \xlra{\om, R} A_j$ fails for all $1\leq i<j\leq K$. 
That is, if they are contained in distinct clusters of the configuration $\om$ \emph{restricted to }$R$.
We say that the subsets $A_1,\dots , A_K$ are \emph{strongly separated} in $R$ 
if $\ovr{A_1}, \dots, \ovr{A_K}$ are separated in $R$.
Here we have abusively used the notation $\ovr{A_i}$ for the set 
$$\ovr{A_i} = \{(u,v) \in V_\slab : \exists v' \in V_S \text{ such that } (u,v') \in A_i\}.$$
In other words, the sets are strongly separated if there is no open path in the rectangle 
whose projection on $\Lat$ crosses the projection of two distinct sets. 

It is easy to check that, if $\om$ is a configuration containing $K$ strongly separated vertical crossings of some rectangle $R$, 
then $$H_{\calC_h( R )} (\omega) \geq K-1.$$ 
Indeed, let $A_1, \dots, A_K$ be vertical crossings of $R$, strongly separated in $R$ in the configuration $\om$. 
In particular $\ovr{A_1}, \dots, \ovr{A_K}$ are disjoint connected sets crossing $R$ vertically. 
Hence we may order them from left to right -- we will assume this is already the case. 
Fix a self-avoiding path $\ga$ contained in $R$ and crossing it horizontally; orient it from left to right. 
It intersects each set $\ovr{A_i}$ at least once. 
But for each $i$, since $A_i$ and $A_{i+1}$ are strongly separated in $R$, 
$\ga$ must contain at least one $\om$-closed edge 
between any point of intersection with $\ovr{A_i}$ and the first following intersection with $\ovr{A_{i+1}}$.
This implies that $\ga$ contains at least one closed edge in the region between $\ovr{A_i}$ and $\ovr{A_{i+1}}$ for every $i$, 
hence at least $K-1$ closed edges overall. This implies the desired bound. 

\begin{proof}[Lemma~\ref{lem:manycrossor}]
	Fix $n$ and $I$ satisfying the assumptions of the lemma and set $v = \frac{1}{100 I}$ 
	(we will specify the values of the constants $c_0$ and $c_1$ later in the proof; 
	it will be apparent that they do not depend on $n$ and $I$). 
	
	In light of the above observation, to prove Lemma~\ref{lem:manycrossor} 
	we aim to show the existence of $2^I$ strongly separated vertical crossings of $\ovr{[0,2n] \times [0,n/2]}$.
	The proof follows the lines of~\cite[Lemma 4.3]{DumMan14} 
	with the essential difference that the crossings need to be strongly separated rather than simply separated. 
	
	We start of with series of claims for $\phi_{p_{3}}$, similar to those in the proof of~\cite[Lemma 4.3]{DumMan14}. 
	In the present context, the proof of these claims will require the gluing lemma \ref{lem:gluing}\ref{lem:gluing1}. 
	Once the claims established, we use them to show that, with positive $\phi_{p_3}$-probability, 
	there exist $2^I$ separated crossings of  $\ovr{[0,2n] \times [0,n/2]}$. 
	Finally, we deduce that there exist $2^I$ strongly separated crossing 
	with positive $\phi_{p_2}$-probability, using Lemma~\ref{lem:bnddistoevent}.
	
	In what follows, the constant $c > 0$ is that of~\eqref{eq:gluing1};
	it only depends on $p_3$ and $S$.
	Define
	\begin{align}\label{eq:max}
		\alpha = \sup \Big\{ \phi_{p_3}\big( \calC_h(\lceil(2+v)k \rceil, 2k) \big) : k\in[\tfrac n8,\tfrac n2]\Big\}.
	\end{align}
	
	\setcounter{claim}{-1}
	\begin{claim}\label{claim:combinecross}
		For $\alpha$ defined as above, we have
		\begin{equation}\label{eq:bound alpha}
			\alpha 
			\le \frac{1}{c}  \Big[\phi_{p_3}(\calC_h(2n,n)) \Big]^{v/28} 
			\le\frac{1}{c} \Big[\phi_{p_3}(\calC_h(2n,n)) \Big]^{2c_1/I},
		\end{equation} 
		where and $c_1 = 1/5600$ 
		(this is constant $ c_1$ that appears in Lemma \ref{lem:manycrossor}). 
	\end{claim}
	
	\noindent{\em Proof of Claim~\ref{claim:combinecross}.}
		Choose $k\in[\frac n8,\frac n2]$ achieving the maximum in~\eqref{eq:max}. 
		We will show by induction on $j \geq 1$ that 
		\begin{align*}
			\phi_{p_3}\big[ \calC_h((2+j v)k, 2k) \big] \geq (c\alpha)^{2j}.
		\end{align*}
		Applying this to $j = 14/v$, we obtain 
		\begin{align*}
			\phi_{p_3}\big( \calC_h(2n, n) \big)
			\geq \phi_{p_3}\big( \calC_h(16k, 2k) \big) \geq (c\alpha)^{28 /v} = (c\alpha)^{2800I}, 
		\end{align*}
		which implies \eqref{eq:bound alpha} readily.

	    \begin{figure}
	      \begin{center}
	        \includegraphics[width=1.0\textwidth]{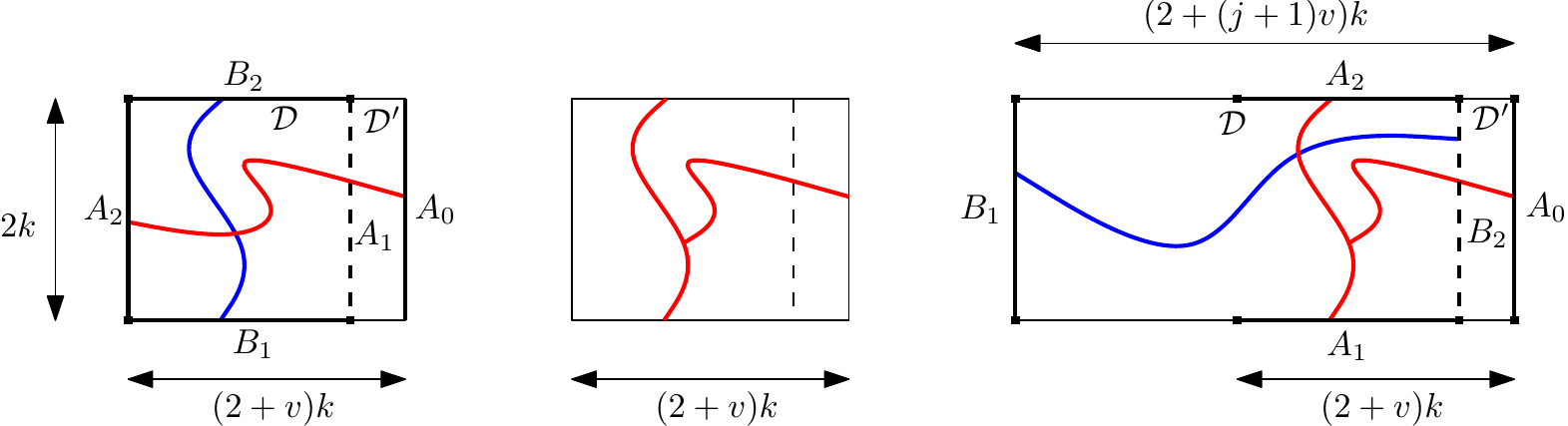}
	      \end{center}
	      \label{fig:claim0}
	      \caption{
	      \emph{Left:} The first application of the gluing lemma allows us to obtain the event $H$.
	      \emph{Middle:} The event $H$; the vertical crossing is contained in the square to the left.
	      \emph{Right:} The second application of the gluing lemma allows us to 
	      	combine $H$ with a horizontal crossing of $[0,(2+jv)k]\times [0,2k]$
		    to generate a horizontal crossing of the longer rectangle $[0,(2+(j+1)v)k]\times [0,2k]$.}
	    \end{figure}
	
		For $j = 1$ the statement is a direct consequence of the definition of $\alpha$. 
        Suppose the statement holds for some $j \geq 1$. 
		Let $H$ be the event that there exists an open cluster in 
		$\ovr{[jvk, (2+ (j+1) v)k]}$ $\ovr{\times [0, 2k]}$ which intersects $\ovr{\{(2+(j+1)v)k\}\times}\ovr{ [0, 2k]}$
		and contains a vertical crossing of the rectangle $\ovr{[jvk, (2+ j v)k] \times [0, 2k]}$.
	
		Apply the gluing lemma \ref{lem:gluing}\ref{lem:gluing1} 
		with domains $\calD' = \ovr{ [jvk, (2+ (j+1) v)k] \times [0, 2k]}$
		and $\calD =  \ovr{ [jvk, (2+ j v)k] \times [0, 2k]}$
		for the two events $\calC_h([jvk, (2+ (j+1)v)k] \times [0, 2k])$
		and $\calC_v([jvk, (2+ j v)k] \times [0, 2k])$
		(i.e. for $A_0 = \{(2+ (j+1)v)k\}\times [0, 2k]$, $A_1 = \{(2+jv)k \} \times [0, 2k]$, $A_2 = \{jvk \} \times [0, 2k]$
		and $B_1 = [jvk, (2+ j v)k] \times \{0\}$, $B_2 = [jvk, (2+ j v)k] \times \{2k\}$). 
		We obtain that 
		$$ 
			\phi_{p_3}( H ) \geq 
			c \phi_{p_3}\big[\calC_h([jvk, (2+ (j+1)v)k] \times [0, 2k]) \big] \phi_{p_3}\big[\calC_v([jvk, (2+ j v)k] \times [0, 2k])\big]
			\geq c \alpha^2. 
		$$
		where the first inequality is the conclusion of the gluing lemma 
		and the second is due to the invariance under translation and rotation. 
		
		Apply now the gluing lemma with the domains
	    $\calD' = \ovr{ [0, (2+ (j+1) v)k] \times [0, 2k]}$
		and $\calD =  \ovr{ [jvk, (2+ j v)k] \times [0, 2k]}$
		for the events $H$ and $\calC_h\big( (2+ (j+1) v)k, 2k\big)$
		(i.e. for $A_0 = \{(2+ (j+1) v)k\}\times [0, 2k]$, 
		$A_1 = [jvk, (2+ j v)k] \times \{0\}$
		$A_2 = [jvk, (2+ j v)k] \times \{2k\}$
		and $B_1 = \{0\} \times [0, 2k]$, $B_2 = \{(2+jv)k \} \times [0, 2k]$).
		The conclusion of the gluing lemma, together with the bound on $\phi_{p_3}(H)$  and the induction hypothesis, yield
		\begin{align*}
			\phi_{p_3}\big[ \calC_h((2+(j+1) v)k, 2k) \big] 
			\geq c \phi_{p_3}[H] \phi_{p_3}\big[ \calC_h([0, (2+ j v)k] \times [0, 2k]) \big]
			\geq (c\alpha)^{2(j+1)}.
		\end{align*}
		which is the desired conclusion.
	\hfill $\diamond$ \vspace{6pt}\\

	Fix an integer $k \in [\frac n4,\frac n2]$ and $u \in [v, 1/12] $ such that $ku \in \bbZ$. 
	The following three claims are concerned with crossings of the rectangle $\bfR(k) = \overline{[-(1 + u)k, (1 +u) k ] \times [0,2k]}$.
	Claims~\ref{startfromcenter}-\ref{onemakesthree} are equivalent to those used in~\cite[Proof of lemma 4.3]{DumMan14}; 
	Claim~\ref{onemakestwostronglydisjoint} however is specific to the case of slabs and requires special attention. 
	We give the proof of all claims for completeness. 
	
	\begin{claim}\label{startfromcenter} 
		Let $\calE(k)$ 
		be the event that there exists a vertical open crossing of $\bfR(k)$,
		with the lower endpoint not contained in $\overline{[-3uk, 3uk] \times \{0\}}$, 
		or the higher endpoint not contained in $\overline{[-3uk, 3uk] \times \{2k\}}$.
		Then $$ \phi_{p_3}(\calE(k)) \leq 4(\alpha + \sqrt{\alpha / c}). $$
	\end{claim}
	
	\noindent{\em Proof of Claim~\ref{startfromcenter}.}
		Let $\beta$ be the $\phi_{p_3}$-probability that there exists a vertical open crossing of $\bfR(k)$,
		with the lower endpoint in $\overline{[-(1 + u)k,-3uk] \times \{0\}}$.    
		By the definition of $\alpha$, the probability of crossing $\overline{[-(1 + u)k, (1 - 2u)k]\times[0,2k]}$ vertically is at
		most $\alpha$.
		Thus, with probability $\beta - \alpha$, there exists a vertical crossing of $\bfR(k)$ 
		with an endpoint in $\overline{[-(1 + u)k,-3uk] \times \{0\}}$ 
		which intersects the vertical line $\overline{\{(1 - 2u)k\} \times [0,2k]}$.
		By reflection with respect to $\{-3uk\}\times [0,2k]$,
		with probability $\beta -\alpha$, there exists an open path in $\overline{[-(1 + 4u)k, (1 - 5u)k] \times [0,2k]}$, 
		between $\overline{[- 3uk ,(1 - 5u)k]\times \{0\}}$ and $\overline{\{-(1 + 4u)k\} \times [0,2k]}$.
		
		When combining the two events above using the first part of the gluing lemma, we obtain    
		$$\phi_{p_3}\big(\calC_h\overline{[-(1 + 4u)k, (1 - 2u)k] \times [0,2k]}\big) \geq c(\beta -\alpha)^2.$$
		The above event has probability less than $\alpha$ (by definition of $\alpha$), hence
		$ \beta \leq (\alpha + \sqrt{\alpha/c}).$
		By considering the other possibilities for the lower and higher endpoints, the claim follows.
	\hfill $\diamond$ \vspace{6pt}\\

	\begin{claim}\label{touchright} 
		Let $\calF(k)$ 
		be the event that there exists a vertical open crossing of $\bfR(k)$
		that does not intersect the vertical line $\overline{\{(1 - 2u)k \} \times [0,2k]}$.
		Then  $$ \phi_{p_3} (\calF(k)) \leq 2\alpha. $$
	\end{claim}
	
	\noindent{\em Proof of Claim~\ref{touchright}.}
		Any vertical crossing of $\bfR(k)$ not intersecting $\overline{\{(1-2u)k \} \times [0,2k]}$ 
		is either contained in  $\overline{[-(1 + u)k, (1 - 2u) k ] \times [0,2k]}$ 
		or in $\overline{[(1 - 2u)k, (1 + u) k ] \times [0,2k]}$.
		Both these rectangles are crossed vertically with probability less than $\alpha$, and the claim follows. 
	\hfill $\diamond$ \vspace{6pt}\\

	\begin{claim}\label{topbeforeright} 
		Let $\calG(k)$ 
		be the event that there exists an open path in $\overline{\bbR \times [0,(2- 11 u)k]}$
		between $\overline{[-3uk, 3uk] \times \{0\}}$ and the vertical segment $\overline{\{(1-2u)k\} \times  [0,(2- 11 u)k]}$.
		Then  $$ \phi_{p_3} (\calG(k)) \leq \alpha + \sqrt{\alpha/c}.$$
	\end{claim}
	
	\noindent{\em Proof of Claim~\ref{topbeforeright}.}
		Let $\beta = \phi_{p_3} (\calG(k))$. 
		Suppose $\calG(k)$ occurs and let $\ga$ be an open path in $\overline{\bbR \times [0,(2- 11 u)k]}$
		between $\overline{[-3uk, 3uk] \times \{0\}}$ and $\overline{\{(1-2u)k\} \times  [0,(2- 11 u)k]}$.
		There are two possibilities for $\ga$. 
		Either $\ga$ crosses the line $\overline{\{-(1-8u)k\} \times [0,(2-11u)k]}$, or it does not. 
		
		The first situation arises with probability at most $\alpha$
		since it induces a horizontal crossing of the rectangle $\overline{[-(1-8u)k, (1-2u)k] \times [0, (2 - 11u)k]}$.
		
		Thus the second situation arises with probability at least $\beta - \alpha$. 
		Then, by symmetry with respect to $\{3uk\} \times \bbR$, 
		with probability at least $\beta - \alpha$ there exists an open path connecting 
		$\overline{[3uk, 9uk] \times \{0\}}$ to $\overline{\{-(1-8u)k\} \times  [0,(2- 11 u)k]}$. 
		Hence, by the first part of the gluing lemma, $\overline{[-(1-8u)k, (1-2u)k] \times [0, (2-11u)k]}$ 
		is crossed horizontally with probability no less than $c \: (\beta - \alpha)^2$. 
		This is less than or equal to $\alpha$ by its definition, and the claim follows.
	\hfill $\diamond$ \vspace{6pt}\\

	In the claims above we have introduced the events $\calE(k)$, $\calF(k)$ and  $\calG(k)$. 
	In addition, define $\widetilde{\calG}(k)$ as the symmetric of $\calG(k)$ with respect to the line $\bbR \times \{ k \}$, 
	i.e. the event that there exists an open path in $\overline{\bbR \times [11uk,2k]}$
	between $\overline{[-3uk, 3uk] \times \{2k\}}$ and $\overline{\{(1-2u)k\} \times  [11uk,2k]}$.
	The bound of Claim~\ref{topbeforeright} applies to $\widetilde{\calG}(k)$ as well. 
	
	All four events revolve around the rectangle ${\bfR(k)}$.
	In the following, we will use translates of these events (by $z\in \Lat$), 
	and we will say for instance that $\calE(k)$ occurs in some rectangle ${\bfR(k)} + z$ 
	if $\calE(k)$ occurs for the translate of the configuration by $-z$. 
	
	\begin{claim}\label{onemakesthree}
		Except on an event $\calH(k)$, with $\phi_{p_3}(\calH(k)) \leq \frac{96} {u} \sqrt{\alpha/c}$, 
		any open vertical crossing of $\bfS(k) = \overline{[0,2n] \times [-k, k]}$, 
		contains two separated vertical crossings of $ \bfS((1 - 11u)k) =  \overline{[0,2n] \times [-(1 - 11u)k, (1 - 11u)k]}$.
		\end{claim}
		
		\noindent{\em Proof of Claim~\ref{onemakesthree}.}
		The rectangle $\ovr{[0,2n] \times [-k, k]}$ is the union of 
		${\bfR_j} = \ovr{[juk,(2 + (j + 2)u)k]}$ $\ovr{\times [-k, k]}$, for $0\le j\le J$, 
		where 
		$$J ~:=~  \big\lfloor\tfrac1u (\tfrac{n}{k} -2)\big\rfloor - 2 ~\leq~ 6/u.$$
		Let $\calH(k)$ be the union of the following events for $0\le j\le J$: 
		\begin{itemize}[nolistsep,noitemsep]
		\item the rectangle $\ovr{[juk,(2 + (j + 1)u)k] \times [-k, k]}$ contains a horizontal open crossing,
		\item $\calE(k)$ occurs in the rectangle $\bfR_j$,
		\item $\calF(k)$ occurs in the rectangle ${\bfR_j}$,
		\item at least one of $\calG(k)$ and $\tilde{\calG}(k)$ occurs in the rectangle ${\bfR_j}$.
		\end{itemize}
		Using a simple union bound and the estimates of Claims~\ref{startfromcenter}-\ref{topbeforeright}, we obtain 
		\begin{align}\label{eq:1001}
			\phi_{p_3}( \calH(k) ) 
			\leq \frac{96 \sqrt{\alpha/c}}{u}. 
		\end{align}
		Consider a configuration not in $\calH(k)$ containing a vertical open crossing $\ga$ of $\bfS(k)$. 
		We are now going to explain why such a crossing necessarily contains two separated crossings of ${\bfS((1-11u)k)}$. 
		
		Since none of the rectangles $\ovr{[juk,(2 + (j + 1)u)k] \times [-k, k]}$ is crossed horizontally,
		$\ga$ is contained in one of the rectangles ${\bfR_j}$.
		Fix the corresponding index $j$. Parametrize $\ga$ by $[0,1]$, with $\ga_0$ being the lower endpoint. 
		
		Since $ \calE(k)$ does not occur in ${\bfR_j}$, $\ga_0$ and $\ga_1$, are contained in
		$\ovr{[(1+(j-2)u)k,(1+(j+4)u)k]}$ $\ovr{ \times \{-k\}}$ and  $\ovr{[(1+(j-2)u)k,(1+(j+4)u)k] \times \{k\}}$, respectively. 
		Moreover, since $\calF(k)$ does not occur in $\bfR_j$, 
		$\ga$ crosses the vertical line $\ovr{\{(2 + (j-1)u)k \} \times [-k, k]}$. 
		Let $t$ and $s$ be the first and last times that $\ga$ intersects this vertical line. 
		
		Since $\calG(k)$ does not occur in $\bfR_j$, 
		$\ga$ intersects the line $\ovr{[0,2n] \times \{(1-11u)k \}}$ before time $t$. 
		Likewise, since $\tilde{\calG}(k)$ does not occur, 
		$\ga$ intersects the line $\ovr{[0,2n] \times \{-(1-11u)k \}}$ after time $s$. 
		This implies that $\ga$ contains at least two disjoint crossings of ${\bfS((1 - 11u)k)}$.
		Call $\ga^1$ the first one (in the order given by $\ga$) and $\ga^2$ the last one. 
		
		The above holds for any vertical crossing $\ga$ of  $\bfS(k)$, 
		hence the crossings $\ga^1$ and $\ga^2$ are necessarily separated in ${\bfS((1 - 11u)k)}$.
		Indeed, if they were connected inside  ${\bfS((1 - 11u)k)}$, then $\calF(k)$ would occur.
	\hfill $\diamond$ \vspace{6pt}\\

	\begin{claim}\label{onemakestwostronglydisjoint}
		Let $\calI(k)$ be the event that there exists an open vertical crossing of $\bfS(k)$, 
		which does not contain two strongly separated vertical crossings of $\bfS((1 - 11u)k)$. Then 
		$$\phi_{p_2}(\calI(k)) \leq \frac{C'} {u} \sqrt{\alpha},$$ 
		where $C' > 0$ is a constant depending only on $p_2$ and $p_3$. 
	\end{claim}
	
	\noindent{\em Proof of Claim~\ref{onemakestwostronglydisjoint}.}
		Let $\om \in \calI(k) \setminus \calH(k)$ and $\ga$ be an $\om$-open vertical crossing of $\bfS(k)$ 
		which does not contain two strongly separated vertical crossings of $\bfS((1 - 11u)k)$. 
		Let $\ga^1$ be the first subpath of $\ga$ crossing $\bfS((1 - 11u)k)$ vertically, and 
		let $\ga^2$ be the last (when $\ga$ is oriented from bottom to top). 
		By choice of $\om$ in $\calI(k) \setminus \calH(k)$, 
		$\ga^1$ and $\ga^2$ are separated in $\bfS((1 - 11u)k)$, but not strongly separated. 
		Hence there exists a third open path $\chi$ in $\bfS((1 - 11u)k)$ that overlaps with both $\ga^1$ and $\ga^2$. 
		Fix such a path $\chi$ and overlap points $u,v \in \Lat$ between $\chi$ and $\ga^1$ and $\ga^2$, respectively. 
		Then, if $\om'$ is the configuration obtained form $\om$ by opening all the edges in $\ovr{\{u\}}$ and $\ovr{\{v\}}$,
		we have $\om'\in \calH(k)$. 
		Indeed, the $\om'$-open path obtained by following $\ga^1$ up to $u$, 
		then $\chi$ to $v$ and finally $\ga^2$ from $v$ to the top of $\bfS((1 - 11u)k)$ crosses $\bfS((1 - 11u)k)$ vertically, 
		but does not contain two sub-paths separated in $\bfS((1 - 11u)k)$.
		Thus 
		$$ \om \in  \big\{H_{\calH(k)} \leq 2|E_S| \big\},$$
		and consequently $\calI(k)\subset \{H_{\calH(k)} \leq 2|E_S| \}$. 
		Lemma~\ref{lem:bnddistoevent} implies that
		\begin{align*}
			\phi_{p_2}\big(\calI(k)\big) 
			\leq \phi_{p_2}\big( H_{\calH(k)} \leq 2|E_S|\big) 
			\leq C^{2|E_S|} \phi_{p_3}\big(\calH(k)\big),
		\end{align*}
		where $C= \frac{q^2(1-p_2)}{(p_3-p_2)[p_2+q(1-p_2)]}.$ 
		By inserting the bound \eqref{eq:1001} on $\phi_{p_3}(\calH(k))$ into the above, 
		we obtain the desired result with $C' = \frac{96 C^{2|E_S|}}{\sqrt{c}}$. 
	\hfill $\diamond$ \vspace{6pt}\\
	
	Getting back to the proof of the lemma.
	Let $k_i = \lfloor (1 - 22vi)n/2 \rfloor $ for $0\le i\le I$.
	We will investigate vertical crossings of the nested strips $\bfS(k_i) =\ovr{[0,2n] \times [-k_i, k_i]}$.
	Note that ${\bfS(k_0)}$ is contained in a translation of the rectangle $\ovr{[0,2n] \times [0,n]}$, 
	and that ${\bfS(k_I)}$ contains a translation of the rectangle $\ovr{[0,2n] \times [0,n/2]}$.
	
	Fix a sequence $(u_i)_i$, 
	with $u_i \in [v, 2v]$ and $k_i u_i \in \bbZ$ for $ 0\le i < I$. 
	The existence of $u_i$ is due to the fact that $v \ge \tfrac4n$ (since $I \le n/400$).
	Define the events $\calI(k_i)$ of Claim~\ref{onemakestwostronglydisjoint} for these values of~$u_i$. 
	Except on the event $\bigcup_{i = 0}^{I-1}  \calI(k_i)$, 
	any configuration with a vertical crossing of ${\bfS(k_0)}$ has $2^I$ 
	strongly separated vertical open crossings of ${\bfS(k_I)}$.
	
	By the union bound, claims \ref{claim:combinecross} and \ref{onemakestwostronglydisjoint} 
	and the definitions of $u$ and $v$, we obtain
	\begin{align*}
		\phi_{p_2} \left( \bigcup_{i = 0}^{I - 1}  \calI(k_i) \right)  
		\leq \frac{C'\sqrt{\alpha}}{u}I		
		\le  \frac{100 C' I^2}{\sqrt{c}} \big[ \phi_{p_3}(\calC_h(2n,n)) \big]^{c_1/I}
		\leq \frac{100 C' c_0 }{\sqrt{c}} \delta,
	\end{align*}
	where the last inequality is due to the choice of $I$.
	We may choose $c_0 = \sqrt{c}/(200 C') > 0$, so that the right-hand side is smaller than $\de/2$.  
	But ${\bfS(k_0)}$ is crossed vertically with $\phi_{p_2}$-probability at least $\delta$,
	hence, with $\phi_{p_2}$-probability at least $\delta /2$, 
	${\bfS(k_I)}$ contains $2^I$ strongly separated vertical crossings.
	By the observations made before the proof, we have 
	$$  
		\phi_{p_2}\big[H_{\calC_h (2n,n/2)} \geq 2^{I} -1\big] 
		\geq \phi_{p_2}\big[H_{\calC_h (\bfS(k_I))} \geq 2^{I} -1 \big] 
		\geq \frac{\de}2,
	$$ 
	which directly implies the desired result. 
\end{proof}

\section{Proof of Theorem~\ref{thm:main}}\label{sec:mainproof}

The previous section showed that for $p > \tilde p_c$, 
crossing probabilities in the hard direction for $2n \times n$ rectangles 
are bounded away from $0$, uniformly in $n$. 
The following two results show us that these probabilities actually tend rapidly to $1$
as $n \to \infty$, for any $p > \tilde p_c$.

We start with a lemma taken from~\cite[Cor. 5.2]{DumMan14} and which is valid in all dimensions.
It is an integrated form of the result of \cite{GraGri11}. 
We do not give the proof here, as it is identical to the one in \cite{DumMan14}.

\begin{lemma}\label{cor:gg_applied}
	For any $0 < p < p' < 1$, there exists $c = c(p)>0$ such that, for $n \geq 1$, 
	\begin{align}\label{eq:ggsh2}
		\phi_{p}(\calC_h(2n,n))
		\big( 1-\phi_{p'}(\calC_h(2n,n)) \big)
		\leq \big(\phi_{p'}(0 \lra \pd \Ball_n) \big)^{c(p' - p)}.
	\end{align}
\end{lemma}
The above lemma, along with Proposition~\ref{prop:RSW}, imply that, 
for $p \in (\tilde p_c,p_c)$ (if such a $p$ exists), $ \lim_n \phi_{p}(\calC_h(2n,n)) = 1$. 
The following proposition tells us that, for such a value of $p$, 
$\phi_{p}(\calC_h(2n,n))$ actually converges to $1$ faster than any polynomial.

\begin{proposition}\label{prop:convergesfast}
	Fix $p < p'$ and $\De >0$. Suppose that $lim_{n \rightarrow\infty} \phi_{p}(\calC_h(2n,n)) = 1$. 
	Then, for $n$ sufficiently large, 
	$$\phi_{p'}(\calC_h(2n,n)) \geq 1-n^{-\De}.$$
\end{proposition}

The proof of Proposition~\ref{prop:convergesfast} is based on the following lemma.
\begin{lemma}\label{lem:hamnocross}
	There exists $\beta > 0$ 
	such that, for any $p<p'$ and $N>n$, 
	\begin{align*}
		\frac{1- \phi_{p'}(\calC_h(2N,N))}{ 1-\phi_{p}(\calC_h(2N,N))} \leq
		\exp
		\left(
			-2(p'-p)\frac{N}{n}
			\left[
				\phi_{p}(\calC_h(2n,n))^{2N/n}
				-2 \left\lceil\frac{N}{n}\right\rceil \Big(1-\phi_{p}(\calC_h(2n,n)) \Big)^\beta 
			\right]
		\right).
	\end{align*}
\end{lemma}

\begin{proof}[Lemma~\ref{lem:hamnocross}]
	We prove this lemma by bounding the expected (under $\phi_{p}$) Hamming distance 
	to the decreasing event $\calC_h(2N,N)^c$ and applying~\eqref{eq:hamming_dec_inf}.
	
	The Hamming distance to $\calC_h(2N,N)^c$ is clearly larger or equal to 
	the number of edge-disjoint horizontal crossings of $\ovr{[0,2N]\times[0,N]}$. Thus 
	\begin{align*}
		\phi_{p}(H_{\calC_h(2N,N)^c}) 
		&= \phi_{p}\big(\text{number of disjoint horizontal crossing of } \ovr{[0,2N]\times[0,N]}\big) \\
		&\geq \left\lfloor\frac{N}{n}\right\rfloor \phi_{p}(\calC_h(2N,n)).
	\end{align*}
	The second inequality is due to the fact that the horizontal crossings of rectangles 
	$\ovr{[0,2N]}$ $\ovr{\times[in,(i+1)n]}$ for $0\leq i < \lfloor N/n \rfloor$ 
	are disjoint and to the invariance of the measure under translation.
	
	Let us now bound $ \phi_{p}(\calC_h(2N,n))$ from below. 
	By the same induction as in Claim \ref{claim:combinecross},
	using the quantitative gluing lemma \ref{lem:gluing}\ref{lem:gluing2} 
	$2\lceil\frac{N}{n}\rceil$ times, we obtain 
	\begin{align}\label{duexNn}
		\phi_{p}\big(\calC_h(2N,n)\big) 
		\geq \phi_{p}\big(\calC_h(2n,n) \big)^{2N/n} - 2\left\lceil\frac{N}{n}\right\rceil \big(1-\phi_{p}(\calC_h(2n,n))\big)^\beta,
	\end{align}
	Where $\beta > 0$ is given by Lemma \ref{lem:gluing}\ref{lem:gluing2}.
	Using~\eqref{eq:hamming_dec_inf} and the fact that $ \lceil \frac{N}{n}\rceil \leq \frac{N}{2n}$, the lemma follows.
\end{proof}

\begin{proof}[Proposition~\ref{prop:convergesfast}]
	Fix $p < p'$ and $\De>0$ as in the proposition. Fix $\eps >0 $ such that  $p+\epsilon<p'$.
	We first introduce two increasing sequences $(n_k)_{k\geq k_0} \in \bbN$ and $(p_k)_{k\geq k_0} \in [p,p']$
	such that
	$$\phi_{p_k}(\calC_h(2n_k,n_k))>1-e^{-\epsilon2^{k}}, $$
	(The indices start from $k_0$ only for a mater of a more clear notation.) 
	
	For $k \geq 1$, set $v(k)=(1-e^{-\epsilon2^{k}})^{2\cdot4^k}-2\cdot4^k e^{-\beta  \epsilon 2^k}$. 
	The sequence $v(k)$ tends to $1$ as $k$ tends to infinity, 
	so we may fix an index $k_0$ such that 
	$v(k)>1/2$ for all $k\geq k_0$. 
	Set $p_{k_0}=p$ and choose $n_{k_0} \in \bbN$ such that 
	$\phi_p(\calC_h(2n,n)) > 1-e^{-\epsilon2^{k_0}}$ for all $n\geq n_{k_0}$
	(the choice of $n_{k_0}$ is possible by hypothesis). 
	Now define, for $k \geq k_0$, 
	\begin{align*}
		n_{k+1}&=n_{k} 4^{k},\\
		p_{k+1}&=p_{k}+\frac{\epsilon}{2^{k-1}}.
	\end{align*}
	We will now prove by induction that $\phi_{p_k}(\calC_h(2n_k,n_k))\geq1-e^{-\epsilon2^{k}}$ for all $k \geq k_0$. 
	The statement is true for $k_0$ by choice of $n_0$. 
	Suppose it is true for some $k \geq k_0$. Then, based on the Lemma~\ref{lem:hamnocross},
	\begin{align*}
		&1 -\phi_{p_{k+1}}(\calC_h(2n_{k+1},n_{k+1})) \\
		&\quad  \leq\exp
		\left(
			-2(p_{k+1}-p_k)\frac{n_{k+1}}{n_k}
			\left[
				\phi_{p_k}(\calC_h(2n_k,n_k))^{2\frac{n_{k+1}}{n_k}}
				-2 \left\lceil\frac{n_{k+1}}{n_k}\right\rceil \Big(1-\phi_{p_k}(\calC_h(2n_k,n_k)) \Big)^\beta 
			\right]
		\right)\\
		&\quad \leq \exp\Big[-2\frac{\epsilon}{2^{k-1}}\: 4^{k}\: v(k)\Big] 
		\leq e^{-\epsilon 2^{k+1}},
	\end{align*}
	and the induction is complete. 
	
	By monotonicity of $p \mapsto \phi_p$, 
	we deduce that $1-\phi_{p'}(\calC_h(2n_{k},n_{k})) \leq e^{-\epsilon 2^{k}}$ for all $k \geq k_0$. 
	Since $n_k = n_{k_0} 4^{k_0 + \dots + (k-1)} \leq n_{k_0} 4^{k^2}$, it follows that for all $k \geq k_0$ sufficiently large 
	$ n_k^{\De} \leq e^{\epsilon2^{k}}$, hence 
	$1-\phi_{p'}(\calC_h(2n_{k},n_{k})) \leq n_k^{-\De}$,  which is the desired statement for $n = n_k$.  
	
	It remains to prove the statement for values of $n$ in between the scales $(n_k)_{k\geq k_0}$.
	Fix $n$ such that $n_k<n<n_{k+1}$ for some $k \geq k_0$. Based on~\eqref{duexNn}, we have
	\begin{align*}
		\phi_{p'}(\calC_h(2n,n))\geq \phi_{p'}(\calC_h(2n_{k+1},n_{k})) \geq v(k).
	\end{align*}
	In order to obtain the desired result,
	it suffices to show that $v(k)>1-n_{k+1}^{-\De}>1-n^{-\De}$ for $k$ sufficiently large. 
	Recall that $n_k \leq n_{k_0} 4^{k^2}$.
	As a consequence $2\cdot 4^{k}e^{- \beta \epsilon 2^{k}} \leq \frac{1}{2}(n_{k_0}\: 4^{k^2})^{-\De}$ for sufficiently large $k$. 
	Moreover, we have
	$$ 
		\big(1-e^{-\epsilon 2^{k}}\big)^{2\cdot4^{k}}
		\geq \exp\Big(-\frac{2\cdot 4^{k}}{e^{\epsilon2^{k}}}\Big) 
		\geq \exp\Big(-\frac{1}{2}\big(n_{k_0}\: 4^{k^2}\big)^{-\De}\Big)
		\geq 1-\frac{1}{2}\big(n_{k_0}\: 4^{k^2}\big)^{-\De}
	$$
	for sufficiently large $k$. 
	The first and last inequality are due to the fact that, for sufficiently small $x$, we have $e^{-x}\geq 1-x \geq e^{-2x}$;
	the second inequality comes from direct asymptotic estimates.  
	Hence $v(k)> 1-(n_{k_0}\: 4^{k^2})^{-\De}>1-n_{k+1}^{-\De}>1-n^{-\De}$ for $k$ large enough. 
\end{proof}

We are finally ready to put the different elements together to prove our main result. 

\begin{proof}[Theorem~\ref{thm:main} ]
	Recall the definitions of $p_c$ and $\tilde p_c$ and that we are aiming to prove $p_c \leq \tilde p_c$. 
	We proceed by contradiction and assume $p_c > \tilde p_c$.
	Then there exist parameters $\tilde p_c < p_0 < p_1 < p_2 < p_c$.
	Corollary~\ref{cor:hardcrossbound} implies that $\{\phi_{p_0}(\calC_h(2n,n)):\, n\geq 1\}$ is bounded away from $0$. 
	Since $p_1 < p_c$, $\phi_{p_1}(0 \lra \pd \Ball_n) \to  0$ as $n \to \infty$, 
	and Lemma~\ref{cor:gg_applied} yields
	$$ \phi_{p_1}\big(\calC_h(2n,n) \big) \xrightarrow[n \to \infty]{} 1.$$
	Proposition~\ref{prop:convergesfast} with $\De = 1$ implies that there exists $n_0$ such that for all $n>n_0$ we have
	\begin{align}
		\phi_{p_2}\big(\calC_h(2n,n) \big) \geq 1- \frac{1}{n}. \label{eq:1overn}
	\end{align}
	Recall the exponent $\beta$ appearing in the second part of the gluing lemma
	($\beta$ may be taken small with no loss of generality; will assume $\beta < 1$ for computational purposes).
	Now choose $n_1>n_0$ such that
	$$ 2^{-n_1} + 4\sum_{k \geq n_1} 2^{-\beta^2 k } <1.$$
	
	For $n \geq n_1$, let $H_n$ be the event that $\ovr{[0,2^{n_1}] \times \{0\}}$ is connected to $\ovr{[0,2^{n}] \times \{2^n\}}$
	inside the domain $\ovr{[0,2^n]^2}$. That is 
	$$ 
		H_n = \ovr{[0,2^{n_1}] \times \{0\}} \xlra{\ovr{[0,2^n]^2}}  \ovr{[0,2^{n}] \times \{2^n\}}.
	$$
	
	\begin{figure}
		\begin{center}
			\includegraphics[width=0.6\textwidth]{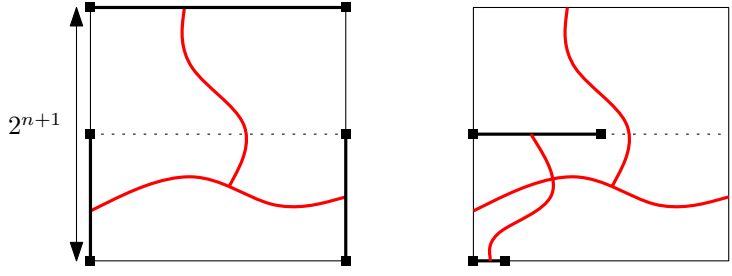}
		\end{center}
		\caption{\emph{Left:} the event $\calA$.
		\emph{Right:} the gluing of $\calA$ et $H_n$ to obtain $H_{n+1}$.}
		\label{fig:finalcombin}
	\end{figure}

%
	Let us estimate the difference between $\phi_{p_2}(H_{n+1})$ and $\phi_{p_2}(H_n)$ for some $n\geq n_1$. 
	Fix such a value $n$ and let $\calA$ be the event that there exists an open crossing in 
	$\ovr{[0,2^{n+1}] \times [0,2^{n}]}$ between $\ovr{\{0\} \times [0,2^{n}]}$ and $\ovr{\{2^{n+1}\} \times [0,2^{n}]}$
	that is connected in $\ovr{[0,2^{n+1}]^2}$ to $\ovr{[0,2^{n+1}]\times \{2^{n+1}\}}$.
	(See the left of Figure \ref{fig:finalcombin} for an illustration.)
	The second part of the gluing lemma and the estimate \eqref{eq:1overn} imply that 
	$$
		\phi_{p_2}(\calA) 
		\geq (1- 2^{-(n+1)})^2 - 2^{-\beta (n+1)}
		\geq 1 - 2^{-n} - 2^{-\beta (n+1)}
		\geq 1 - 2 \cdot 2^{-\beta n}.
	$$ 
	We may now apply the gluing lemma using the events $\calA$ and $H_n$. 
	That is, apply it with 
	$\calD = \ovr{[0,2^{n+1}] \times [0,2^{n}]}$,
	$\calD' = \ovr{[0,2^{n+1}]^2}$,
	$A_0 = \ovr{[0,2^{n+1}] \times \{2^{n+1}\}}$, 
	$A_1 = \ovr{\{0\} \times [0,2^{n}]}$,
	$A_2 = \ovr{\{2^{n+1}\} \times [0,2^{n}]}$,
	$B_1 = \ovr{[0,2^{n_1}] \times \{0\}}$
	and $B_2 =  \ovr{[0,2^{n}] \times \{2^n\}}$.
	Then we obtain 
	\begin{align*}
		\phi_{p_2}\big(H_{n+1}\big) 
		&= \phi_{p_2}\big(B_1 \xlra{\calD'} A_0\big) 
		\geq \phi_{p_2}(H_{n})\phi_{p_2}(\calA)  - \big(1-\phi_{p_2}(\calA)\big)^\beta\\
		& \geq \phi_{p_2}(H_{n}) - 2 \cdot \big(1-\phi_{p_2}(\calA)\big)^\beta \\
		& \geq \phi_{p_2}(H_{n}) - 4 \cdot 2^{-\beta^2 n}.
	\end{align*}
	Finally, as a consequence of \eqref{eq:1overn}, $\phi_{p_2}(H_{n_1}) > 1 - 2^{-n_1}$. 
	We may therefore deduce that, for all $n \geq n_1$,
	$$
		\phi_{p_2}(H_{n}) 
		\geq \phi_{p_2}(H_{n_1}) -  4 \sum_{n_1\leq k < n} 2^{-\beta^2 k}
		\geq 1 - 2^{-n_1} -  4 \sum_{ k \geq n_1} 2^{-\beta^2 k} > 0, 
	$$
	due to our choice of $n_1$.
	Observe now that this implies 
	\begin{align}
		\phi_{p_2}\big(\ovr{[0,2^{n_1}] \times \{0\}} \lra \infty\big) 
		= \lim_{n\to \infty} \phi_{p_2} \big(\ovr{[0,2^{n_1}] \times \{0\}} \lra \pd \La_{2^n}\big)
		\geq   \lim_{n \to \infty} \phi_{p_2} (H_n) >0. 
	\end{align}
	By finite energy property and the FKG inequality, we deduce that
	the origin belongs to an infinite open cluster with positive $\phi_{p_2}$-probability.
	This contradicts the choice of $p_2$, and the theorem is proved.
\end{proof}

\section*{Acknowledgements}

The authors would like to thank Hugo Duminil-Copin for numerous inspiring discussions about this work. 
This research was supported by the NCCR SwissMAP, the ERC AG COMPASP, and the Swiss NSF.

\bibliographystyle{plain} 
\bibliography{biblicomplete}
\bigskip 

\noindent
\begin{minipage}{0.46\textwidth}
	\footnotesize\obeylines
	\textsc{University of Fribourg}
	\textsc{Fribourg, Switzerland}
	\textsc{E-mail:} \texttt{ioan.manolescu@unifr.ch}\smallskip
\end{minipage}\hfill
\begin{minipage}{0.46\textwidth}
	\footnotesize\obeylines
	\begin{flushright}
		\textsc{University of Geneva}
		\textsc{Geneva, Switzerland}
		\textsc{E-mail:} \texttt{aran.raoufi@unige.ch}
	\end{flushright}
\end{minipage}

\end{document}